\documentclass[preprint,12pt]{elsarticle}

\usepackage{type1cm}        

\usepackage{makeidx}         
\usepackage{graphicx}        
\usepackage{multicol}        
\usepackage[bottom]{footmisc}
\usepackage{placeins}
\usepackage{subcaption}
\usepackage{newtxtext} 
\usepackage[monochrome]{xcolor}
\usepackage{amsthm}       %
\usepackage[varvw]{newtxmath}       

\usepackage[overload]{empheq}

\usepackage{comment}

\def\x{{\bf x}}
\def\s{{\bf s}}

\def\bu{{\bf u}}

\def\bff{{\bf f}}

\def\bs{{\bf s}}

\def\bA{{\bf A}}
\def\bR{{\bf R}}
\def\bu{{\bf u}}

\def\bD{{\bf D}}

\def\dsp{\displaystyle}

\def\blue{\textcolor{blue}}
\def\red{\textcolor{red}}

\def\dx{{\rm d}\x}

\def\O{\Omega}
\def\G{\Gamma}

\def\R{\mathbb{R}}

\def\Nn{{\cal N}}
\def\Tt{{\cal T}}
\usepackage{multicol}
\usepackage{multirow}
\def\H10{H^1_{ \partial \O \setminus \partial \O_S}}

\def\hetse{heterogeneities}

\def\dsp{\displaystyle}

\def\nodes{{\cal V}}

\def\uharm{u_{\Delta}}
\def\uloc{u_b}

	\def\order{p}
	\def\sig{\partial \Omega_j \setminus \partial \Omega_S}
	
\def\VHp{V_{H,p}}
\def\VHpG{\VHp^\Gamma}

\def\smallop{\frac{1}{20}\mathcal{H}_j}

\usepackage{amsthm}       %
\newtheorem{proposition}{Proposition}[section]
\newtheorem{corollary}{Corollary}[section]
\newtheorem{lemma}{Lemma}[section]

\newtheorem{remark}{Remark}[section]
\newtheorem{theorem}{Theorem}[section]
\makeindex             

\journal{Applied Numerical Mathematics}

\begin{document}

\begin{frontmatter}


\title{Robust Methods for Multiscale Coarse Approximations of Diffusion Models in Perforated Domains}


\author[uca]{Miranda Boutilier\corref{cor1}}
\ead{miranda.boutilier@univ-cotedazur.fr}
\author[uca,inria]{Konstantin Brenner}
\ead{konstantin.brenner@univ-cotedazur.fr}
\author[uca,strath]{Victorita Dolean}
\ead{work@victoritadolean.com}

 \cortext[cor1]{Corresponding Author}

\affiliation[uca]{organization={Université Côte d'Azur,  Laboratoire J.A. Dieudonné, CNRS UMR 7351},
            city={Nice},
            postcode={06108}, 
            country={France}}

\affiliation[inria]{organization={Team COFFEE, INRIA Centre at Université Côte d'Azur},
             city={Nice},
            postcode={06108},
            country={France}}

 \affiliation[strath]{organization={University of Strathclyde, Department of Mathematics and Statistics},
            city={Glasgow},
            postcode={G1 1XH},
  country={U.K.}}

\begin{abstract}
 For the Poisson equation posed in a domain containing a large number of polygonal perforations, we
propose a low-dimensional coarse approximation space based on a coarse polygonal partitioning of the domain.
Similarly to other multiscale numerical methods, this coarse space is spanned by locally discrete harmonic basis functions.
Along the subdomain boundaries, the basis functions are piecewise polynomial.  The main contribution of this article is an error estimate 
regarding the $H^1$-projection 
over the coarse space; this error estimate depends only on the regularity of the solution over the edges of the coarse partitioning.
For a specific edge refinement procedure, the error analysis establishes superconvergence of the method even if the true solution has a low general regularity. \red{Additionally, this contribution numerically explores the combination of the coarse space with domain decomposition (DD) methods. This combination leads to an efficient two-level iterative linear solver which reaches the fine-scale finite element error in few iterations. }
It also bodes well as a preconditioner for Krylov methods and provides scalability with respect to the number of subdomains. 
 
\end{abstract}



\end{frontmatter}

\section*{\textcolor{violet}{Nomenclature}}

\textcolor{violet}{
\begin{tabular}{p{2cm}|p{10cm}}
Parameter Notation & Parameter Definition \\
\hline 
$\left( D_j \right)_{j= 1,\ldots, N}$ & Nonoverlapping partitioning of polygonal domain $D$ \\
$\O_S$ & Union of all perforations \\
$\left( \O_j \right)_{j= 1,\ldots, N}$ & Nonoverlapping partitioning of $\O= D \setminus \overline{\O_S}$ \\ 
 $\left( \O_j'\right)_{j= 1,\ldots, N}$ & Overlapping partitioning such that $\O_j \subset \O_j'$. \\
$\Gamma$ & Coarse skeleton $\bigcup_{j = 1, \ldots, N } \partial \O_j\setminus \partial\O_S$ \\
$\left(e_k\right)_{k = 1, \ldots, N_e}$  & Nonoverlapping partitioning of $\G$ into coarse edges  \\
$\nodes$ & Set of coarse grid nodes $ \bigcup_{k = 1, \ldots, N_e } \partial e_k$ \\
$H$  &  Max length over all coarse edges, $\max_{k = 1,\ldots, N_e} |e_k|$. \\ 
$\mathcal{H}_j$ & 
Maximum (length or width) distance between minimal and maximal 
coordinates of $\O_j$  \\
$N_\O$ & Set of internal triangulation nodes \\
$\bA$ & $N_\O\times N_\O$ matrix derived from finite element triangulation of Poisson equation \\ 
$\bR_H$ & Discrete coarse matrix of harmonic basis functions \\ 
$\bR_j'$ &  Boolean restriction matrices for $\O_j'$  \\
$\bR_j'$ &  Boolean restriction matrices for $\O_j'$  \\
$\overline{\bR}_j$  &  Boolean restriction matrices for the closure  of $\O_j$, given by $\overline{\O}_j$
\end{tabular}}

	\section{Introduction and Model Problem}
	\label{sec:1}
	
	Numerical modeling of overland flows plays an increasingly important role in predicting, anticipating, and controlling floods. Anticipating these flood events can aid in the positioning of protective systems including dams, dikes, or rainwater drainage networks. One of the challenges of this numerical modeling of urban floods is that small structural features (buildings, walls, etc.) may significantly affect the water flow. In
particular, the representation of the structural features in the numerical model impacts both the
timing and extent of the flood \cite{YuLane1, YuLane2, Abily}. 
 
	Modern terrain survey techniques including photogrammetry and Laser Imaging, Detection, and Ranging (LIDAR) allow us to acquire high-resolution topographic data for urban areas.
	The data set used in this article has been provided by Métropole Nice Côte d'Azur (MNCA) and allows for
  an infra-metric description of the urban geometries \cite{buildings}. From the hydraulic perspective,  these structural features can be assumed to be essentially impervious,  and therefore represented as perforations (holes) in the model domain.


	In this article, we consider a linear diffusion model. Let $D$ be an open simply connected polygonal domain in $\mathbb{R}^2$. We denote by $\left(\O_{S,k}\right)_k$ a finite family of perforations in $D$ such that each $\O_{S,k}$ is an open connected polygonal subdomain of $D$. The perforations are mutually disjoint, that is $\overline{\O_{S,k}}\cap \overline{\O_{S,l}} = \emptyset$ for any $k\neq l$.  We denote $\O_S = \bigcup_k \O_{S,k}$ and $\O = D\setminus \overline{\O_S}$, assuming that the family $\left(\O_{S,k}\right)_k$ is such that $\O$ is connected.  Note that the latter assumption implies that $\O_{S,k}$ are simply connected.
	
	We are interested in the boundary value problem given by
	\begin{equation}\label{model_pde}
		\left\{
		\begin{array}{rll}
			- \Delta u &=& f \qquad \mbox{in} \qquad \O, \\
			\dsp  \frac{\partial u}{\partial \bf{n}} &=& 0 \qquad \mbox{on} \qquad \partial  \O  \cap \partial \O_S,\\
			u &=& 0 \qquad \mbox{on} \qquad  \partial \O \setminus \partial \O_S,\\
		\end{array}
		\right.
	\end{equation}
	where $\mathbf{n}$ is the outward normal to the boundary and $f\in L^2(\O)$.
 
	Let us denote by $(\cdot, \cdot)_{L^2(\O)}$ and $(\cdot, \cdot)_{H^1(\O)}$ the standard $L^2$ and $H^1$ scalar products, that is,
	$$
	(u,v)_{L^2(\O)} = \int_\O u v \, \dx \qquad \mbox{and} \qquad 
	(u, v)_{H^1(\O)} = \int_\O \nabla u\cdot \nabla v \, \dx.
	$$
	Setting 
	\begin{equation}\label{eq:h1partial}
		\H10(\O)= \{ u\in H^1(\O) \,  | \, u |_{\partial \O \setminus \partial \O_S} = 0\},
	\end{equation}
	the weak solution of \eqref{model_pde} satisfies the following variation formulation: Find $u\in  \H10(\O)$ such that
	\begin{equation}\label{model_weak}
		(u, v)_{H^1(\O)}  = (f, v)_{L^2(\O)} \qquad \mbox{for all} \qquad v\in \H10(\O).
	\end{equation}
	
	In the context of urban flood modeling, $u$  represents the flow potential (pressure head) and  $\left(\Omega_{S,k}\right)_k$ can be thought of as a family of impervious structures such as buildings,  walls, or other similar structures.  Although the linear model \eqref{model_pde} is overly simplified to be directly used for urban flood assessment, the more general nonlinear elliptic or parabolic models are common in free surface flow simulations.  Such models arise from Shallow Water equations either by neglecting the inertia terms \cite{diffwave} or within the context of semi-implicit Froude-robust time discretizations \cite{lowfroude}.

	
Depending on the geometrical complexity of the computational domain, the numerical resolution of \eqref{model_pde} may become increasingly challenging. 
The typical model domain (see for example Figure \ref{fig:framesols}) resulting from the realistic
description of the urban environment will contain numerous perforations that are represented on
different scales.  \red{The resolution of these structures can result in extremely small computational elements. Thus, we wish to employ numerical strategies that consider two levels of space discretization (see Figure \ref{fig:triang} for a visual of both levels).  Specifically, we employ a multiscale method to solve \eqref{model_pde} on these urban geometries, 
aiming to achieve computational efficiency compared to classical fine-scale solution methods. Like other multiscale methods, this also allows for parallel implementation to further improve efficiency.}

In this contribution, we investigate two numerical strategies that are capable of handling the multiscale features of the urban geometries. As mentioned, both strategies consider two levels of discretization in space.
 The first level of space discretization is 
based on a coarse polygonal partitioning of $\Omega$, while the second is associated with the fine-scale triangulation and is designed to resolve the small-scale details of the model domain.  The coarse partitioning is used to decompose the solution of the problem \eqref{model_pde} into the sum of the locally harmonic component and local subdomain contributions, where the subdomain contributions can be efficiently computed in parallel.  This splitting leads to a system that can be seen as a continuous version of the Schur complement problem. 
We then introduce a low-dimensional space, called here the Trefftz or discrete Trefftz space, that serves to approximate the locally harmonic component of the solution.  This coarse approximation space
is built upon basis functions that satisfy the local Laplace problems (either exactly or via a finite element approximation) and have polynomial traces along the boundaries of the coarse partitioning. Both continuous and discrete variants of the coarse space are discussed.



	\begin{figure}
		\centering
		\begin{subfigure}{0.49\textwidth}
			\centering
			\includegraphics[height=5cm, width=6.3cm]{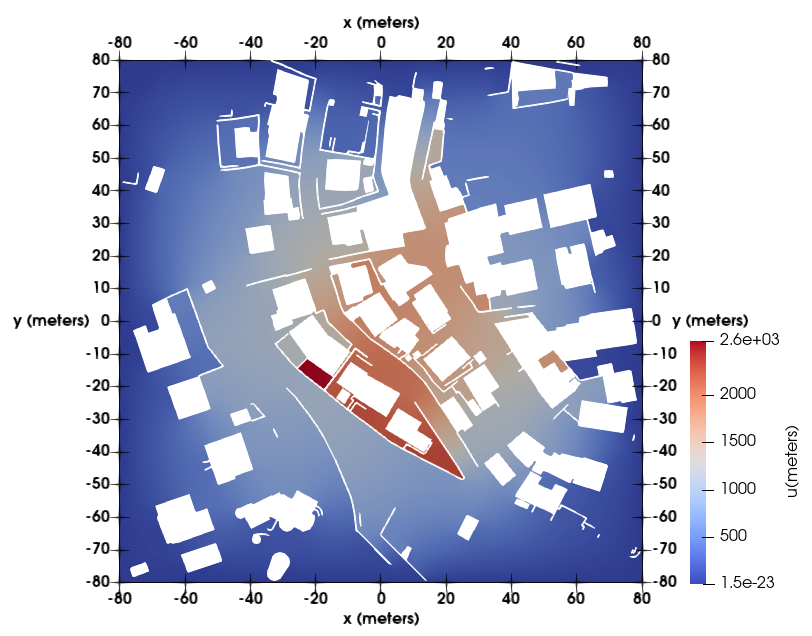}
		\end{subfigure}
		\begin{subfigure}{0.49\textwidth}
			\centering
			\includegraphics[height=5cm, width=6.1cm]{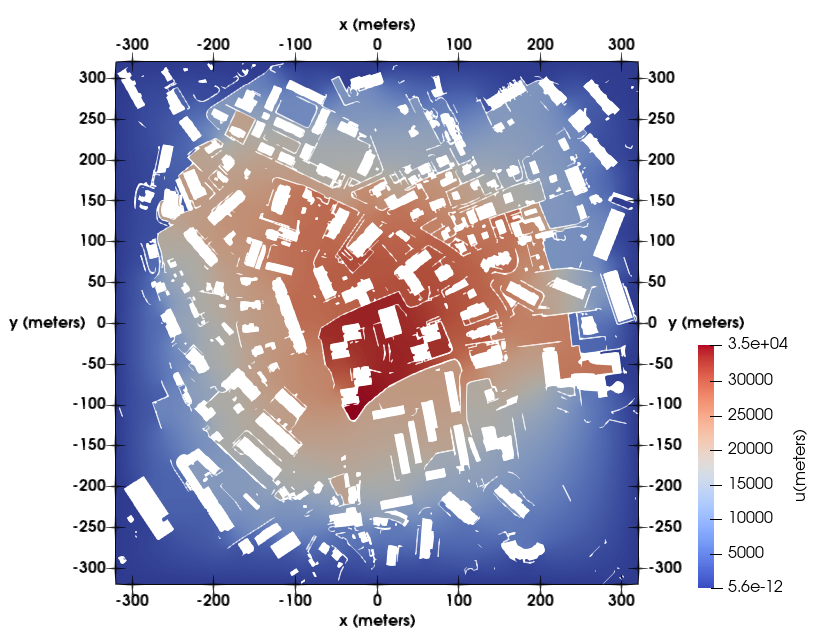}
		\end{subfigure}
		\caption{Finite element solution with $f = 1$ on model domains based on a smaller (left) and larger (right) data sets.}
		\label{fig:framesols}
	\end{figure}

	\begin{figure}
		\centering
		\begin{subfigure}{0.49\textwidth}
			\centering
			\includegraphics[height=5.cm, width=6.5cm]{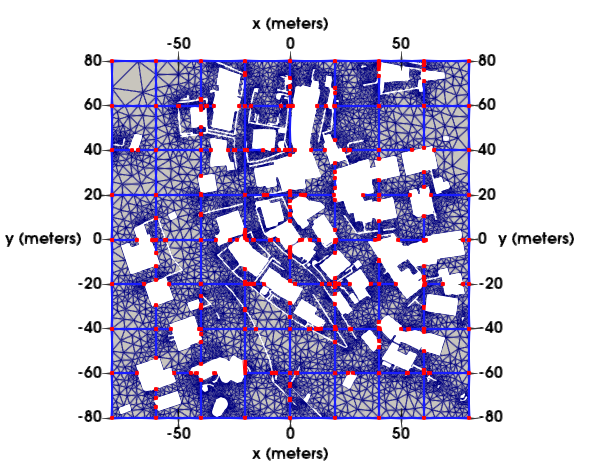}
		\end{subfigure}
		\begin{subfigure}{0.49\textwidth}
			\centering
			\includegraphics[height=5cm, width=6cm]{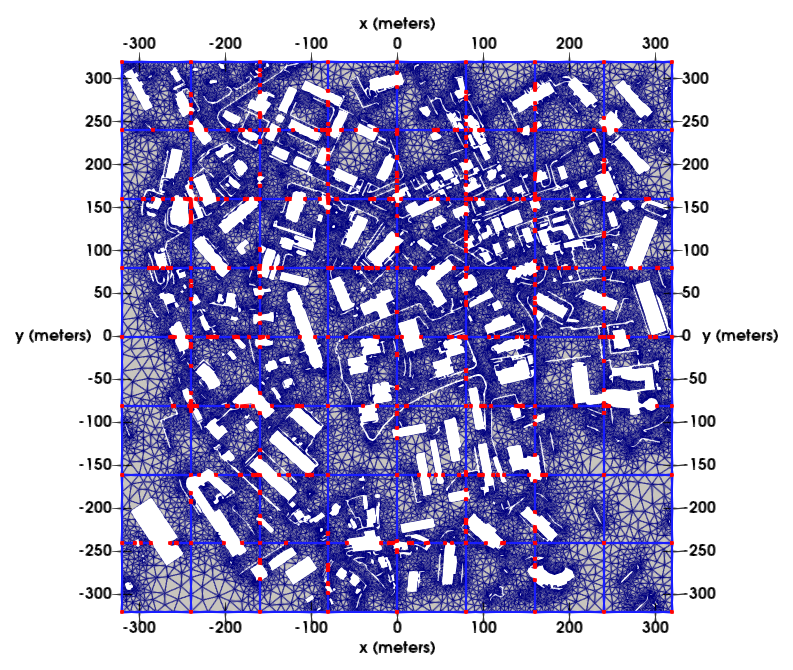}
		\end{subfigure}
		\caption{Coarse (thick lines) and fine (thin lines) discretizations for smaller (left) and larger (right) data sets, with the coarse nodes shown by red dots. }
		\label{fig:triang}
  \end{figure}

The first contribution of this article concerns the
ability of the Trefftz space to approximate the locally harmonic component of the solution (or the solution itself). 
We present an a priori estimate in the energy norm which, for a specific edge refinement procedure, 
results in superconvergence of the method regardless of the general regularity of the solution. 
This theoretical finding is confirmed by numerical experiments.


There exists extensive literature on problems involving highly oscillatory coefficients or multiscale geometrical features, including for example the Localized Orthogonal Decomposition (LOD) method \cite{lod} and Heterogeneous Multiscale methods (HMM) \cite{hmm}. Additionally, to solve these heterogeneous problems, there exist multiscale partition of unity methods such as
 the Generalized Finite Element Method (GFEM) \cite{gfem} and the Multiscale Spectral Generalized Finite Element Method (Ms-GFEM) \cite{msgfem, msgfem2}. 
In \cite{msgfem, msgfem2}, the optimal approximation spaces are constructed locally by solving generalized eigenvalue problems over a set of overlapping subdomains. 
The partition of unity functions are used to glue the local contributions into a globally conforming coarse approximation space.
 We note that these methods vary from the approach
presented here
 as we consider nonoverlapping coarse cells and do not rely on eigenproblems.
 
Since the perforations in our model problem are structured, we can exploit the structural features of the \hetse. The method we choose to consider here is closer to 
the Multiscale Hybrid-Mixed Method (MHM) \cite{mhm},
the Multiscale Finite Element Method (MsFEM) \cite{origmsfem}, or polytopal methods such as Virtual Element Methods (VEM) \cite{vem}.
In MsFEM, the domain is partitioned into coarse cells, allowing the method to take advantage of parallel computing. 
Basis functions are then computed numerically
based on the PDE on each coarse cell;
in this way, our methodology is similar to MsFEM. In comparison to the classical MsFEM, our method leads to a larger coarse space.
 Compared to VEM, the major difference is that we numerically compute the approximation of the locally harmonic basis.  By doing so, we manage to incorporate singular functions (corresponding to the corners of the domain) into the coarse space. We are also able to deal with very general polygonal cells (not star-shaped,  not simply connected,  etc.),  which differs from VEM.
Additionally, we willingly avoid using a method of fundamental solutions of any kind, because of our long-term motivation in problems more complex than the linear model problem \eqref{model_pde}.
The MHM method combines ideas from the Mixed Finite Element Method (MFEM) and MsFEM. 
Compared to MHM, where Neumann traces are approximated by piecewise polynomials, our method is formulated in terms of approximate Dirichlet traces. \blue{We remark as well the existence of Generalized Multiscale Finite Element Methods (GMsFEM) \cite{gmsfem}, which were created to generalize MsFEM methods by constructing coarse trial spaces via a spectral decomposition of snapshot spaces}.

The second contribution of the article is the combination of 
the proposed coarse approximation with local subdomain solves in a two-level domain decomposition (DD) method.  The idea of DD methods is to partition the domain into multiple subdomains and solve on each subdomain locally, gluing the local solutions together at the end to obtain a global solution. This results in an efficient iterative  solver for the algebraic system resulting from the fine-scale finite element method.  The obtained algorithm can also be thought as a way of improving the precision of the original coarse approximation in the spirit of iterative multiscale methods (see e.g.  \cite{imsfvm}).  
	We note that furthermore,  the discrete Trefftz coarse space can also be combined with domain decomposition methods to provide a two-level  preconditioner for Krylov methods. The coarse component provides robustness of
the preconditioned Krylov method in terms of iteration counts with respect to the number of subdomains. It also provides an additional acceleration when compared to the iterative method.
	
	
	Specifically on perforated domains with small and numerous perforations, the authors of \cite{legoll} introduced an MsFEM method for diffusion problems with
 Dirichlet boundary conditions imposed on the perforation and domain boundaries, with an  error estimate provided. In \cite{msfembubble, legoll2}, this method was extended to advection-diffusion problems, with \cite{legoll2} imposing both Dirichlet and Neumann boundary conditions on the perforation boundaries.
 In \cite{legoll, legoll2, msfembubble}, Crouzeix-Raviart type boundary conditions were imposed on the local problems to provide 
robustness with respect to the position of the perforations. As well, the addition of bubble functions is included.
 The authors of \cite{taralova} also provides an MsFEM method for perforated domains, providing numerical results and analysis for problems posed on domains with numerous small, regular perforations. In \cite{taralova}, Neumann boundary conditions are posed on the perforation boundaries, with the assumption
that the coarse grid does not intersect with the perforations along the edges for analysis purposes (although this assumption is not necessary numerically).
 \blue{Aside from classical MsFEM methods, the authors of \cite{hmmperf} introduced a HMM method to solve elliptic homogenization problems in perforated
domains, with periodic perforations required for analysis purposes. Additionally, the authors of \cite{gmsfemperf, gmsfemperf2, gmsfem3} proposed a GMsFEM method for perforated domains for small, regular perforations. }
	
While literature on DD approaches to solve model problems in perforated domains is limited, the model problem can be thought of as the extreme limit case of the elliptic model containing highly contrasting coefficients with zero conductivity on the perforations. 
	Two-level domain decomposition methods have been extensively studied for such heterogeneous problems.  
	There are many classical results for coarse spaces that are constructed to resolve the jumps of the coefficients; see \cite{analysis, multileveldiscon, balancing}   for further details.
	Robust coarse spaces have constructed using the ideas from
	MsFEM in \cite{multiscalehighaspectratio, multiscalepdes}. 
	Approaches to obtain a robust coarse space without careful partitioning of the subdomains include spectral coarse spaces such as those given in \cite{eigenproborig,dtn, 2012spillane,geneo}.
	The authors of \cite{quasimonotone} discuss these varied coefficient problems in the case where the coarse grid is not properly aligned with the \hetse. The combination of spectral and MsFEM methods can be found in \cite{shemorig}, where the authors enrich the MsFEM coarse space with eigenfunctions along the edges. Additionally, the family of GDSW
	(Generalized Dryja, Smith, Widlund) 
	methods \cite{gdsw}  employ energy-minimizing coarse spaces and can be used to solve heterogeneous problems on less regular domains. These spaces are discrete in nature and involve both edge and nodal basis functions.
	To deal with coefficient jumps in highly heterogeneous problems, an adaptive GDSW coarse space was introduced in \cite{adaptivegdsw}.  
	
 The present article expands upon two published conference papers by the same authors \cite{dd27, fvca}. 
Compared to the previous works, we consider a higher order Trefftz space and provide a detailed proof for the error estimate.
In addition, we provide extended numerical experiments regarding the convergence of a two-level domain decomposition method based on the  Trefftz coarse space. 
The article is laid out as follows. In Section \ref{sec:1lvl}, we introduce the continuous Trefftz coarse space and discuss its approximation properties. 
In Section  \ref{sec:2lvl}, we introduce the Trefftz coarse space in its discrete matrix form. With this, we provide the matrix forms of the coarse approximation and domain decomposition methods. The two-level domain decomposition methods are presented as both an iterative solver and a preconditioner for Krylov methods. In Section \ref{sec:num}, we provide numerical results for the Trefftz space used as a coarse approximation, in an iterative Schwarz method, and as a preconditioner for Krylov methods. We provide numerical results for two different types of model domains; the first domain is a simplified domain with one perforation at the corner, and the second type of model domain is a realistic urban domain with numerous perforations of various shape.  Section \ref{sec:conc} concludes with a summary and brief description of future work.
	
\section{Continuous Trefftz Approximation}\label{sec:1lvl}


In this section, we introduce the splitting of \eqref{model_weak}, inspired by the Schur method, and propose the continuous finite-dimensional coarse space that can be used to efficiently approximate the locally harmonic component of the solution.
We perform the error analysis of the coarse Galerkin approximation. In this regard, the main results are Proposition \ref{prop_iterpolation_error} and Theorem \ref{thm:thm}.
 
\subsection{Coarse Mesh and Space Decomposition}

We begin with a coarse discretization of $\O$ which involves a family of polygonal cells $\left( \O_j \right)_{j= 1,\ldots, N}$, the so-called coarse skeleton $\Gamma$, and the set of coarse grid nodes that will be referred to by $\nodes$. 

The construction is as follows.  Consider a finite nonoverlapping polygonal partitioning of $D$ denoted by $\left( D_j \right)_{j= 1,\ldots, N}$ and an induced nonoverlapping partitioning of $\O$ denoted by $\left( \O_j \right)_{j= 1,\ldots, N}$  such that $\O_j=D_j \cap \O$. We will refer to  $\left( \O_j \right)_{j= 1,\ldots, N}$  as the coarse mesh over $\O$.
	Additionally,  we denote by $\Gamma$ its skeleton,  that is  $\Gamma=  \bigcup_{j = 1, \ldots, N } \partial \O_j\setminus \partial\O_S$.

We define $H^1_\Delta(\Omega)$ as a subspace of $H^1(\O)$ composed of piece-wise harmonic functions, weakly satisfying the homogeneous Neumann boundary conditions on $\partial \O \cap \partial \O_S$ such that
 \begin{equation}  
H_\Delta^1(\O)= \{\, u \in H^1(\O)  \,\, | \,\, (u|_{\O_j}, v)_{H^1(\O_j)} =0  \quad \text{for all} \quad v \in H^1_{\partial \O_j\setminus \partial \O_S}(\O_j)\}.
\end{equation} 
 In other words, $H_\Delta^1(\O)$ can be defined as $u \in H^1(\O)$ such that for all subdomains $\O_j$, the equations
\begin{equation}
		\left\{
		\begin{array}{rll}
			- \Delta u|_{\O_j} &=& 0 \qquad \mbox{in} \qquad \O_j, \\
			\dsp  \frac{\partial u}{\partial \bf{n}} &=& 0 \qquad \mbox{on} \qquad \partial  \O_j  \cap \partial \O_S,\\
		\end{array}
		\right.
	\end{equation}
	are satisfied in a  weak sense.  We further define the space $H^1_\G(\Omega)$ as the subspace of functions vanishing on the coarse skeleton $\G$ such that 
 \begin{equation}  
H_\G^1(\O)= \{\, u \in H^1(\O)  \,\, | \,\, u|_{\Gamma} =0 \}.
\end{equation} 
By definition, $H^1_\Delta(\O)$ is orthogonal to $H^1_\G(\O)$. Since $H^1_\G(\O)$ is a closed subspace of $H^1(\O)$, we deduce that $H^1(\O) = H^1_\Delta(\Omega) \oplus H^1_\Gamma(\Omega)$ (see e.g. \cite{func}). In other words, a given function $v\in H^1(\O)$ admits a unique decomposition into $v_\Delta + v_b$,  where $v_\Delta \in H_\Delta^1(\O)$, $v_b \in H^1_\G(\O)$ and $(v_\Delta, v_b)_{H^1(\O)} = 0$. Although for  simplicity we will call $v_\Delta$ the ``locally harmonic'' or ``piece-wise'' harmonic component of $v$,  we wish to stress that the space $H^1_\Delta(\O)$ also contains information about the normal traces of $v_\Delta$ over $\partial \O_S$.  The function $v_b$ will be referred to as the local or ``bubble'' component of $v$.

	Using the orthogonal decomposition of $H^1(\O)$ introduced above, we can express \eqref{model_weak} as 
	the following Schur complement problem: Find $u = \uharm + \uloc$ with $\uharm \in H^1_\Delta(\O) \cap H^1_{\partial \O \setminus \partial \O_S}(\Omega)$ and $u_b \in H^1_{\Gamma}(\Omega)$ satisfying
	\begin{align}[left=\empheqlbrace]
			( \uharm, v)_{H^1(\O)} &= (f, v)_{L^2(\O)}  \qquad \forall  v \in H^1_\Delta(\Omega) \cap H^1_{\partial \O \setminus \partial \O_S}(\Omega), \label{weak_HP_1}
			\\
			( \uloc, v)_{H^1(\O)} &=  (f, v)_{L^2(\O)} \qquad  \forall  v\in H^1_{\Gamma}(\Omega).\label{weak_HP_2}
\end{align}

	We remark that the formulation \eqref{weak_HP_1}-\eqref{weak_HP_2} uncouples the local and the piece-wise harmonic components of $u$,  moreover that the ``bubble'' component of the solution $\uloc$ can be computed from  \eqref{weak_HP_2}
	locally (and in parallel) on each $\Omega_j$,  while the problem \eqref{weak_HP_1} remains
	globally coupled over the computational domain $\Omega$.

	\subsection{Continuous Trefftz Space}\label{sec:contandapprox}
	
	We now proceed with the goal to approximate the locally harmonic component $\uharm$ of the solution. For this, we introduce the Trefftz coarse space, a finite-dimensional subspace of $H^1_\Delta(\Omega)$ that is spanned by functions that are piece-wise polynomial on the skeleton $\G$.

 Let $\left(e_k\right)_{k = 1, \ldots, N_e}$ denote a nonoverlapping partitioning of $\G$ such that each ``coarse edge''  $e_k$ is an open planar segment, 
and we denote $H = \max_{k = 1,\ldots, N_e} |e_k|$.
The set of coarse grid nodes is given by 
$\nodes= \bigcup_{k = 1, \ldots, N_e } \partial e_k$.
Figure \ref{fig:coarsedofs} illustrates the location of the nodal degrees of freedom that typically result from clipping $(D_j)_j$ with $\O_S$.  
It is important to note that a straight segment of $\Gamma$ may be subdivided into multiple edges. 
As we will show in Section \ref{sec:num}, this subdivision can be intentional (see Figure \ref{fig:edgeref}) to achieve convergence of the coarse approximation. 

We define  
$$\VHpG= \{\,v \in C^0(\overline{\Gamma}) \,\, | \,\, v|_{e_k} \in \mathbb{P}_\order(e_k)  \quad \text{for all} \quad k = 1, \ldots, N_e \},$$
where $\mathbb{P}_\order(e_k)$ denotes the set of polynomials of order (at most $\order$) over an edge $e$. 
 We also define 
 $$\VHp= \{\,v \in H_\Delta^1(\O) \,\, | \,\, v|_\Gamma \in \VHpG \}.$$

	
	
	\begin{figure}
		\centering
		\includegraphics[height=4.5cm, width=6.5cm]{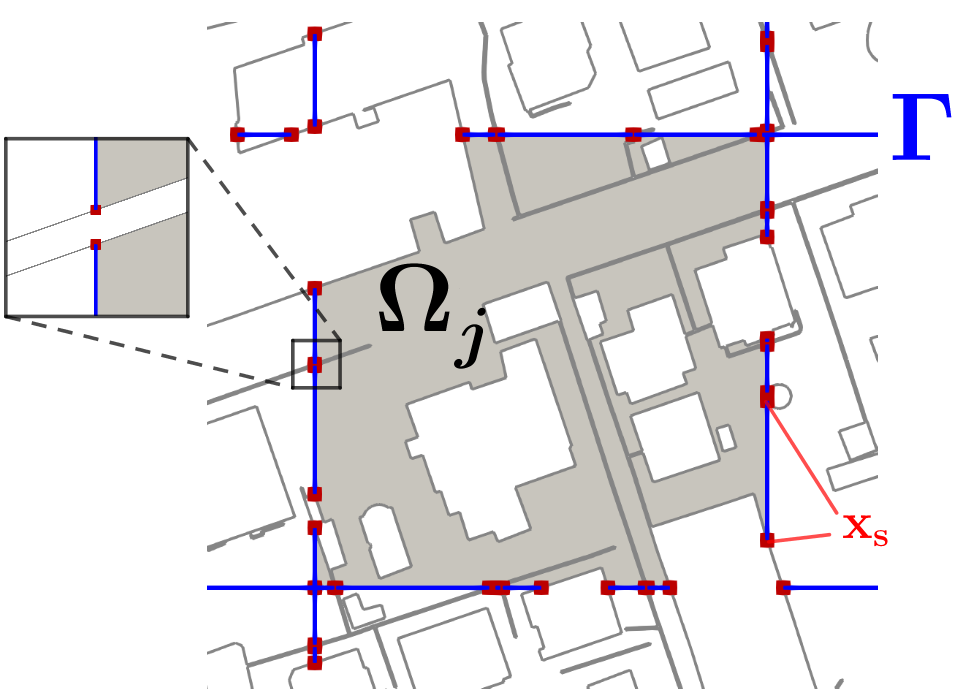}
		\newline
		\caption{ Coarse grid cell $\O_j$, nonoverlapping skeleton $\Gamma$ (blue lines), and coarse grid nodes $\mathbf{x}_s=(x_s,y_s) \in \nodes$ (red dots). Coarse grid nodes are located at $\overline{\Gamma} \cap \partial \O_S.$}
		\label{fig:coarsedofs}
	\end{figure}	

Let $(g_\alpha)_{ \alpha=1, \ldots, N_{H,p}   }$ be some basis of $\VHpG$, where $N_{H,p}$ is the dimension of $\VHpG$.   In practice,  $(g_\alpha)_{ \alpha=1, \ldots, N_{H,p}  }$ is set up by combining the set of the nodal piece-wise linear ``hat'' functions with the set of the higher order edge-based basis functions.  

The function $\phi_\alpha \in \VHp$ associated to $g_\alpha$ is computed by weakly imposing 
	\begin{equation}\label{eq:laplacedir}
		\left\{
		\begin{array}{rllll}
			\Delta \phi_\alpha &=& 0  &  \mbox{in} &  \O_j, \\
			\dsp 	\frac{\partial \phi_\alpha}{\partial \bf{n}}&=&0 &\mbox{on} & \partial \O_j \cap \partial \O_S,\\
			\phi_\alpha &=& g_\alpha &\mbox{on} &  \partial \O_j \setminus \partial \O_S,\\
		\end{array}
		\right.
	\end{equation}
 \blue{against the tests functions in $V_{h,p}.$}



	\blue{The Galerkin method to approximate $u_\Delta \in H_\Delta^1(\O)$} based on the coarse space reads as follows: find $u_{\Delta,H} \in V_{H,p} \cap H^1_{\partial \O \setminus \partial \O_S}(\Omega)$ such that
	\begin{equation}\label{eq:femvariational}
		(u_{\Delta,H}, v)_{H^1(\O)} = (f, v)_{L^2(\O)} \qquad \forall v \in V_{H,p} \cap H^1_{\partial \O \setminus \partial \O_S}(\Omega).
	\end{equation}

	\subsection{Error Analysis}
	We now provide an error estimate for the convergence of the Trefftz approximation space. We begin with a few preliminary results that will be necessary in the proof.
	
	\begin{lemma}[Polynomial interpolation over an edge]\label{lem_interp_1d}
		Let $e_k$ be a coarse edge. We denote by ${\cal I}^\order_k$ the Lagrange interpolator of order $\order$ defined with respect to some set of interpolation points containing the endpoints of $e_k$.  
		Then,  there exists $c_0 = c_0(p)>0$ such that for every $v$ of sufficient regularity, we have
		$$
		\|v - {\cal I}_k^\order v\|_{L^2(e_k)} +  |e_k| | v - {\cal I}_k^\order v|_{H^1(e_k)} \leq c_0 |e_k|^{\order+1} | v |_{H^{\order+1}(e_k)}.
		$$
This directly implies the following $L^2$ and $H^1$ estimates:
	$$
	\| v - {\cal I}_k^p v \|_{L^2(e_k)} \leq c_0 H^{\order+1} | v |_{H^{\order+1}(e_k)}
	\qquad
	\mbox{and}
	\qquad
	\| v - {\cal I}_k^p v \|_{H^1(e_k)} \leq c_1 H^{\order} | v |_{H^{\order+1}(e_k)},
	$$
	where $c_1  = c_0 \left( 1 + H^2 \right)^{1/2}$ can be bounded e.g.  as $c_1  \leq c_0 \left( 1 + {\rm diam}(\Omega)^2 \right)^{1/2}$.

	\end{lemma}
	\begin{proof}
		See Proposition 1.5 and 1.12 of \cite{feinterp}. See also Lemma 4.2 and Remark 4 of \cite{hpvem} for estimates with explicit dependency on $\order$ and higher dimension.
	\end{proof}

	
	\begin{lemma}[Gagliardo-Nirenberg interpolation inequality]\label{lem:gag}
		Let $\Omega$ be Lipschitz domain in $\mathbb{R}^d$
		and $\|\cdot\|_{s,q}$ denote the Slobodeskii-Sobolev norm in $W^{s,q}(\Omega)$, $s\geq 0,q\geq 1$ (see the reference below).
		Let $s,s_1,s_2\geq 0$, $1\leq q, q_1, q_2 \leq \infty$ and $\theta \in (0,1)$ be such that
		$$
		s = \theta s_1  + (1-\theta) s_2, \quad \frac{1}{r} = \frac{\theta}{q_1}+\frac{1-\theta}{q_2}.
		$$
		Then,  there exists $c_{GN} = c_{GN}(s_1, q_1, s_2, q_2)>0$ such that 
		$$
		\|u\|_{s,q} \leq c_{GN} \|u\|_{s_1,q_1}^\theta \|u\|_{s_2,q_2}^{1 - \theta},
		$$
		for all $u \in  W^{s_1, q_1}(\O) \cap W^{s_2,q_2}(\O)$
		as long as the following condition fails:
		$s_2$ is an integer $\geq 1$, $q_2 = 1$ and $s_2-s_1 \leq 1- 1/q_1$.
	\end{lemma}
	\begin{proof}
		See Theorem 1 of \cite{gagagain} for details.
	\end{proof}
	
	\begin{corollary}[Interpolation in $H^s$]\label{cor:cor}
		Let $\|\cdot\|_s$, $s\geq 0$ denote the $H^s(\O)=W^{s,2}(\O)$ norm, where $\O$ satisfies the assumptions of Lemma \ref{lem:gag}. Let  $0 \leq \theta \leq 1$ and $s = \theta s_1  + (1-\theta) s_2$. There exists $c_{GN} = c_{GN}(s_1, s_2)>0$ such that 
		$$
		\|u\|_{s} \leq c_{GN} \|u\|_{s_1}^\theta \|u\|_{s_2}^{1 - \theta}, 
		$$
		for all $u \in  H^{\text{max} (s_1, s_2)}(\O).$
		
		\begin{proof}
			The result follows from Lemma \ref{lem:gag}, setting $q_1=q_2=2$ and using standard Sobolov embedding theory. 
		\end{proof} 
	\end{corollary}

    \begin{lemma}\label{lemma_extension}
    Let $1/2 < s\leq 1$ and let $g \in H^s(\sig)$ satisfy $g|_{\partial (\O_j \cap \G)} = 0$. Then, the function
    $\widetilde{g}: \partial D_j \rightarrow \R$ defined by
    $$
    \widetilde{g}(x) = \left\{\begin{array}{lll}
        g(x) & x\in \partial (\O_j \cap \G), \\
        0 & x \in \partial D_j \setminus \partial (\O_j \cap \G),
    \end{array}
            \right.
    $$
    belongs to $H^s(\partial D_j)$.
    \end{lemma}
    \begin{proof}
    The proof is quite basic, for the full proof we refer e.g. to Lemma 5 of \cite{brenner2016gradient}.
    \end{proof}

 	\begin{lemma}\label{lemma_lifting}
 	Let $g \in H^1_0(\G_j)$, then there exists $\phi \in H^1(\O_j)$ satisfying $\phi|_{\G_j} = g$ such that
 	$$
 	|\phi|_{H^1(\O_j)} \leq c_{GN} C_j \| g \|^{1/2}_{L^2(\G_j)} \| g \|^{1/2}_{H^1(\G_j)} ,
 	$$
 	where $C_j$ depends only on $D_j$.	
 	\end{lemma}
 	\begin{proof}
 	In view of Lemma \ref{lemma_extension}, we can extend $g$ on $\partial D_j$ by zero so that the extension $\widetilde{g}$ is in $H^1(\partial D_j)$.  Then, there exists $\widetilde{\phi} \in H^1(\O_j)$ such that $\widetilde{\phi} = \widetilde{g}$ on $\partial D_j$ and 
 	$$
 	| \widetilde{\phi}  |_{H^1(D_j)} \leq C_j  \| \widetilde{g} \|_{H^{1/2}(\partial D_j)}.
 	$$
 	Using Sobolev interpolation, we have
  			\begin{align}
 	| \widetilde{\phi}  |_{H^1(D_j)} 
 	& \leq  c_{GN} C_j  \| \widetilde{g} \|^{1/2}_{L^2(\partial D_j)} \| \widetilde{g} \|^{1/2}_{H^1(\partial D_j)},  \nonumber \\
 	 &	=  c_{GN} C_j  \| g \|^{1/2}_{L^2(\G_j)} \| g \|^{1/2}_{H^1(\G_j)}.  \nonumber 
 	\end{align}
 	The result follows by setting $\phi = \widetilde{\phi}|_{\O_j}$.
\end{proof}


With preliminary lemmas established, we present the following proposition, which is a novel contribution of this article.

\begin{proposition}[Interpolation error estimate]\label{prop_iterpolation_error}
Let $\left( \gamma_l \right)_{l= 1, \ldots,  N_\gamma}$ be a finite nonoverlapping partitioning of $\G$ and $v$ the element of $H^1_{\Delta}(\O)$ such that the traces of $v$ belong to $H^{\order+1}(\gamma_l)$ for all $l$ and some $p = 1, 2, \ldots $.  Assume in addition that the set of coarse edges $ \left( e_k\right)_k$ is a subdivision of $ \left( \gamma_l\right)_l$.  Then,  there exists $\phi \in H^1(\O)$ such that $\phi|_\G \in \VHpG$ and satisfying 
		\begin{equation}\label{H1_est}
			| v - \phi |_{H^1(\Omega)} \leq C H^{ \order + \frac{1}{2}} 
			\left( \sum^{N_\gamma}_{l=1}  | v |^2_{H^{\order+1}(\gamma_l)} \right)^{1/2}.
		\end{equation}			
 \end{proposition}

 	\begin{proof}
  Since $u \in H^1(\gamma_l)$, it follows that $u\in C^0(\overline{\gamma_l})$. However, being in $H^{1/2}(\G)$, $u$ can not have  discontinuities. Therefore $u\in H^1(\G)$ and thus is continuous on $\overline{\G}$.
  
  Let ${\cal I}^\order_\Gamma$ be an interpolation operator from $C^0(\Gamma)$ to $\VHpG$
  such that ${\cal I}^\order_\Gamma v|_{e_k} = {\cal I}^\order_k v$ for all $k$,  where ${\cal I}^\order_k$ is defined in Lemma \ref{lem_interp_1d}.  Let $E_\Gamma =  u|_\Gamma - {\cal I}^\order_\Gamma (v|_\Gamma)$. Since $E_\Gamma(x) = 0$ for every $x \in {\cal V}$,  it follows from Lemma \ref{lemma_lifting} that
 there exists $\psi_j \in H^1(\O_j)$ satisfying $\psi_j|_\Gamma = \widetilde{E}_\Gamma$ such that 
        $$
        |  \psi_j |_{H^1(\O_j)} \leq c_{GN} C_j \| E_\Gamma \|^{1/2}_{L^2(\G_j)} \| E_\Gamma \|^{1/2}_{H^1(\G_j)}.
        $$
 Let $\phi \in H^1(\O)$ be defined as $\phi|_{\O_j} = \uharm|_{\Omega_j} - \psi_j |_{\Omega_j}$ for all $\O_j$.  We have
        \begin{align}
        | v - \phi|_{H^1(\O_j)} 
        & \leq c_{GN} C_j \| E_\Gamma \|^{1/2}_{L^2(\G_j)} \| E_\Gamma \|^{1/2}_{H^1(\G_j)}, \nonumber \\
        & \leq c_{GN} C_j  \left( \sum_l ||E_{\G,j}||^2_{L^{2}(\gamma_{l} \cap \G_j ) } \right)^{1/4} \left( \sum_l \|E_{\G,j}\|^2_{H^{1}(\gamma_{l} \cap \G_j) } \right)^{1/4}. \nonumber        
        \end{align}
  Since $(e_k)_k$ is a subdivision of $(\gamma_l)_l$, we deduce from Lemma \ref{lem_interp_1d}
  \begin{equation}\label{eq_estimate_H12D}
  | v - \phi|_{H^1(\O_j)} 
  \leq  c_{GN} c_0 c_1 C_j \,  H^{\order+1/2} \left( \sum_l | u |^2_{H^{\order+1}(\gamma_{l} \cap \G_j)} 
  \right)^{1/2}.
    \end{equation} 
     By summing \eqref{eq_estimate_H12D} over all $\O_j$, we obtain \eqref{H1_est} with $C = \sqrt{2} c_{GN} c_0 c_1 \max_j{C_j}$.
\end{proof}

 	\begin{proposition}[Best approximation]\label{prop_best_approx}
 	The solution of \eqref{eq:femvariational} satisfies
		$$
		| u_\Delta - u_{\Delta,H}  |_{H^1(\O)} \leq | u_\Delta - \phi  |_{H^1(\O)},
		$$
		for any $\phi \in H^1(\O)$ such that $\phi|_\G \in \VHpG$.
 	\end{proposition}
 	\begin{proof}
  
 	We first remark that,  because $u_{\Delta,H}$ is the projection of $\uharm$ on $\VHp$,  it satisfies 
\begin{equation}\label{best_approx_VH}
		    | \uharm - u_{\Delta,H} |_{H^1(\Omega)} \leq | u_\Delta - v_H  |_{H^1(\Omega)},
		\end{equation}
		for any $v_H \in V_{H,p}$. Given  some  $\phi \in H^1(\O)$ such that $\phi|_\G \in \VHpG$, we set $v_H \in V_{H,p}^\G$ be such that $v_H|_\G = \phi|_\G$. Since $u_\Delta - v_H \in H^1_\Delta(\O)$,
		 we have that
\begin{equation}\label{best_approx_VH_in_H1}
     |u_\Delta - v_H|_{H^1(\O)} \leq |u_\Delta - \phi|_{H^1(\O)},
     \end{equation}
     that is to say that $u_\Delta - v_H$ is the minimizer
		of $\psi \mapsto | \psi |_{H^1(\O)}$ over the set $\{ \psi \in H^1(\O_j) \, | \,  \psi =  u_\Delta - v_H \mbox{ on } \G_i  \}$.
  The result follows by combining \eqref{best_approx_VH} and \eqref{best_approx_VH_in_H1}.
 	\end{proof}
From the propositions \ref{prop_best_approx} and \ref{prop_iterpolation_error}, we deduce the following error estimate regarding the locally harmonic part of the solution. 

\begin{theorem}\label{thm:thm}
Let $u$ be a solution of \eqref{model_weak} and $\uharm$ its $H^1$ projection on $H^1_\Delta(\O)$.  Let $u_{\Delta,H}$ satisfy \eqref{eq:femvariational}.  Then,  under the assumptions of Proposition \ref{prop_iterpolation_error} with $\uharm$ instead of $v$,
$$
| \uharm - u_{\Delta,H} |_{H^1(\Omega)} \leq C H^{ \order + \frac{1}{2}} 
		\left( \sum^{N_\gamma}_{l=1}  | u |^2_{H^{\order+1}(\gamma_l)} \right)^{1/2}.
$$
\end{theorem}

	\begin{remark}

 \blue{Theorem \ref{thm:thm} expresses the approximation properties of $V_{H,p}$ in $H_\Delta^1(\O)$, or the approximation of $u_\Delta$ by $u_{\Delta,H}$. In addition, the estimate of the theorem holds for $u - (u_{\Delta,H} + u_b)$, where $u_b$ satisfies \eqref{weak_HP_2}. However, we remark that $u_{\Delta,H}$ also provides a low-order approximation of $u$.} Indeed, in view of the orthogonality between $H^1_\Delta(\Omega)$ and $H^1_{\Gamma}(\Omega)$, we have
  $$
  |u-u_{\Delta,H}|^2_{H^1(\O)} =  |u_\Delta-u_{\Delta,H}|^2_{H^1(\O)} + 
  |u_b|^2_{H^1(\O)}.
  $$
  While the first term in the right-hand side can be estimated based on Theorem \ref{thm:thm}, the $H^1$ norm of $u_b$ can be controlled using \eqref{weak_HP_2} and local Poincaré inequalities. More precisely, we have
  $$
  |u_b|_{H^1(\O_j)} \leq C_{P,j} {\rm diam}(\Omega_j) \| f \|_{L^2(\Omega_j)},
  $$
where $C_{P,j}$ denote the Poincaré constants associated to $\O_j$.

	\end{remark}

 \begin{remark}
	The broken $H^{k+1}$ norm in the right-hand side of \eqref{H1_est} involves only the traces of the solution along the sections of the coarse skeleton.  Therefore,  this estimate is valid for $u$ having low general regularity that is due,  for example, to corner singularities.  As a matter of fact,  the estimate \eqref{H1_est} provides an a priori criterion for the adaptation of the coarse mesh: one has to ensure that the edge norm in the right-hand side is small.  For a sufficiently regular right-hand side $f$, this can be achieved by moving the coarse edges away from the ``bad'' perforation corners.  
 \end{remark}

 	We further note that Theorem \ref{thm:thm} is especially valuable for a so-called space or edge refinement,  which is a procedure that involves splitting the edges of an otherwise fixed coarse grid.  In that case, one observes superconvergence of the error with a rate of $k+\frac{1}{2}$.

	\section{Discrete Trefftz Space and Two-level Schwarz Method}\label{sec:2lvl}

We begin with the algebraic formulation and implementation of the problem.
Let us consider a triangulation $\Tt$ of $\O$ which is assumed to be conforming with respect to the  polygonal partitioning  $\left( \O_j \right)_{j= 1,\ldots, N}$ (see Figure \ref{fig:triang}).
We denote by $\Nn_{\overline{\O}}$ the set of the triangulation nodes.
In order to account for Dirichlet boundary condition imposed on $\partial \O \setminus \partial \O_S$, we introduce a set of internal nodes
		$$
		\Nn_{\O} = \{ \ \mathbf{x}_s \in \Nn_{\overline{\O}} | \, \ \mathbf{x}_s \notin \partial \O \setminus \partial \O_S \}.
		$$ 
		The total number of nodes in $\Nn_{\O}$ is denoted by $N_{\O}$.    
   
We denote by $V_h$ the space of functions that are continuous and triangle-wise polynomial on each triangle of $\Tt$, where $h$ denotes the maximal element diameter,  and by
$V_{h,0}$ the  space $V_h \cap \H10(\O)$.  
  The finite element method for the variational problem \eqref{model_weak} consists in finding 
		$u_h \in V_{h,0}(\O)$ such that 
		\begin{equation}
			a(u_h, v_h) = (f, v_h) \qquad \forall \, v_h \in V_{h,0}(\O).
		\end{equation}
		
		Let us denote by $(\eta_\bs)_{\bs \in \Nn_{\overline{\O}} }$ the set of the standard  finite element  basis functions associated to the nodal degrees of freedom.   The associated ``fine-scale'' finite element discretization results in the linear system 
	\begin{equation}\label{Au_f}
		\bA \mathbf{u}=\mathbf{f},
	\end{equation}
		where $\bA$ is a $N_{\O} \times N_{\O}$ matrix with elements 
		\begin{equation*}
			A_{ij} = a(\eta_{\s_j}, \eta_{\s_i})  = \int_\O \nabla \eta_{\s_j} \cdot \nabla \eta_{\s_i} \, \dx,
		\end{equation*}
		and  $\bff$ is the vector from $\R^{N_{\O}}$ with components $ f_j = \int_\O f \eta_{\s_j} \, \dx$.

	Because the triangular mesh resolves the fine-scale structure of the domain,  the system may be quite large and the size of the triangular elements may vary by several orders of magnitude.  As a result, the matrix $\bA$ is expected to be poorly conditioned.

 \subsection{Discrete Trefftz Approximation}
 
    In most practical situations, the coarse basis functions defined by \eqref{eq:laplacedir}  
	can not be computed analytically.
 Therefore, we consider the finite element approximation of $V_{H,p}$.
 The basis of the discrete Trefftz space is obtained through the finite element approximation of \eqref{eq:laplacedir}. 
	The finite element discretization of \eqref{eq:laplacedir} results in the system of the form 
	$$
	 \tilde{\bA}_j \boldsymbol{\phi}_{s}^j=\mathbf{b}_s^j,
	$$
	where $\tilde{\bA}_j $ is the local stiffness matrix and $\mathbf{b}_s^j$ accounts for the Dirichlet boundary data in \eqref{eq:laplacedir}.
\blue{Let $\overline{\bR}_j$ denote the boolean restriction matrices corresponding to the nodes of $\overline{\O_j}$, and let $\boldsymbol{\phi}_s$ be a global vector which, when restricted to a subdomain $\O_j$, returns $\boldsymbol{\phi}_s^j$. That is, 
$$\boldsymbol{\phi}_s=  \sum_{j \in \mathcal{N}_s}  \overline{\bR}_j^T \overline{\mathbf{D}}_j \boldsymbol{\phi}_{s}^j,$$
        for $s=1, \ldots, N_\nodes$, where  $\overline{\mathbf{D}}_j$ are partition of unity matrices corresponding to $\overline{\O_j}$ and
        $\mathcal{N}_s= \{ \,  j \,  \,| \, \, \mathbf{x}_s  \, \, \text{is contained in}  \, \, \O_j  \, \}.$ }
	The discrete Trefftz space is then defined as the
 span of the basis functions $\boldsymbol{\phi}_s,  s=1, \ldots, N_{H,p}$. Discretely, the coarse transition matrix $\bR_H$ is such that the $k$th row of $\bR_H$ is given by $\boldsymbol{\phi}_k^T$ for $k=1, \ldots, N_\nodes$.
 
 
 The finite element counterpart of \eqref{eq:femvariational} and therefore the discrete coarse approximation from Section \ref{sec:contandapprox} can be expressed algebraically as 
 \begin{equation}\label{eq:disccoarseapprox}
     \bu_{\Delta,H}= M_H^{-1} \bff,
 \end{equation}
where $$M_H^{-1}=\bR_H^T(\bR_H\bA \bR_H^T)^{-1}\bR_H.$$	

 \subsection{Discrete Two-level Schwarz Method}
 	Now, we will show how the coarse approximation introduced in the previous section can be combined with the Restricted Additive Schwarz (RAS) method to construct a simple yet efficient iterative linear solver for the fine-scale finite element method. 
  
	Let $\left( \O_j'\right)_{j= 1,\ldots, N}$ denote the overlapping partitioning of $\O$ such that $\O_j \subset \O_j'$.  In practice,  each $\O_j'$ is constructed by propagating $\O_j$ by a few layers of triangles. 



  We denote by $\bR_j'$ the boolean restriction matrices corresponding to the nodes of $\O_j'$
and set $\bA'_j =  \bR_j' \bA \left(\bR_j'\right)^T$. Here, the boolean matrices $\bR_j'$ are of size $(N_{j} \times N_{\O})$, where $N_j$ denotes the number of degrees of freedom in $\O_j'$ and $N_{\O}$ denotes the number of degrees of freedom in $\O$.

Given this framework, we introduce the following iterative procedure. Generally, we take the coarse approximation as the initial iterate $\bu^{0}=\bu_{\Delta,H}$. This allows us to begin with the accuracy of the coarse approximation and continue to iterate to finite element precision. With this, the iteration is given by
	\begin{equation}\label{eq:iterativeras}
		\begin{array}{llll}
			\bu^{n+\frac{1}{2}} &= \bu^{n}+ M_{RAS}^{-1}(\mathbf{f}- \bA \bu^{n}), \\
   	\bu^{n+1} &= \bu^{n+\frac{1}{2}}+ M_{H}^{-1}(\mathbf{f}- \bA \bu^{n+\frac{1}{2}}),\\
		\end{array}
	\end{equation}
	where  $$\dsp M_{RAS}^{-1}=\sum_{j=1}^{N} \left(\bR_j'\right)^T\bD_j(\bA'_j)^{-1}\bR'_j,$$ and
	 $\bD_j$ denote the partition-of-unity matrices satisfying 
$$I_N= \sum_{j=1}^{N} (\bR_j')^T \bD_j (\bR_j'),$$ where $I_N$ is the identity matrix of size $(N_{\O} \times N_{\O})$.

	Additionally, we introduce the preconditioned system
	\begin{equation}\label{eq:precond}
		M_{RAS,2}^{-1}\bA \bu= M_{RAS,2} ^{-1}\mathbf{f},
	\end{equation}
	where the two-level discrete Restricted Additive Schwarz preconditioner is given additively by 
	$$	M_{RAS,2} ^{-1}= \bR_H^T(\bR_H\bA \bR_H^T)^{-1}\bR_H+ \sum_{j=1}^{N} \left(\bR_j'\right)^T\bD_j(\bA'_j)^{-1}\bR'_j.$$
	We expect the use of a well-constructed preconditioner to provide acceleration in terms of Krylov iteration count. 
 Generally, using preconditioned Krylov methods produces faster convergence than the iterative solver, both in terms of computation time and iterations.


 We note that iteration \eqref{eq:iterativeras} can be written as 
	\begin{equation}
		\begin{array}{llll}
   	\bu^{n+1} &= \bu^{n}+ (M_{H}^{-1}+M_{RAS}^{-1}( \mathbf{I}-\bA M_{H}^{-1} )   )(\mathbf{f}- \bA \bu^{n}),\\
		\end{array}
	\end{equation}
a hybrid method. We implement this hybrid iterative method as an additive method will not converge and a hybrid method is necessary for convergence. 
 However, we choose to use the additive two-level RAS preconditioner \eqref{eq:precond} for Krylov methods as it is commonly used in domain decomposition literature and has an increased capacity for parallel computing. With this, we remark that we could use a hybrid method as a Krylov accelerator and that the Schwarz preconditioners presented here are not exhaustive.

 In the context of domain decomposition, our coarse approximation is referred to as a ``coarse space" for the Schwarz methods. However, we remark that the construction of the matrix form of the coarse approximation and coarse space are identical, with the difference being solely in application.
	


	\section{Numerical Results}\label{sec:num}
	
	In this section, we illustrate the performance of the discrete Trefftz space in three different scenarios involving either a standalone Galerkin approximation \eqref{eq:femvariational}, an iterative approach \eqref{eq:iterativeras}, or a preconditioner approach \eqref{eq:precond}.

	\subsection{Coarse Approximation on L-Shaped Domain}
	To properly display the approximation properties, it is helpful to have an exact solution to the equation. However, the generation of this exact solution is difficult with multiple singularities. To combat this issue, we can use a model domain with one singularity/corner in the domain; an example of this would be an L-shaped domain with a reentering corner (Figure \ref{fig:edgeref}).  The domain is defined by $D = (-1,1)^2$,   $\O_S = (0,1)^2$ and $\O = D \setminus \overline{\O_S}$.  We consider the problem \eqref{model_pde} with zero right-hand side and a non-homogeneous Dirichlet boundary condition on $\partial \O \setminus \partial \O_S$ provided by the singular exact solution $u(r,\theta)= r^{\frac{2}{3} } \cos( \frac{2}{3}( \theta - \pi/2))$. 
	
	\begin{figure}
		\centering
		\begin{subfigure}{0.49\textwidth}
			\centering
			\includegraphics[height=4.cm, width=4.5cm]{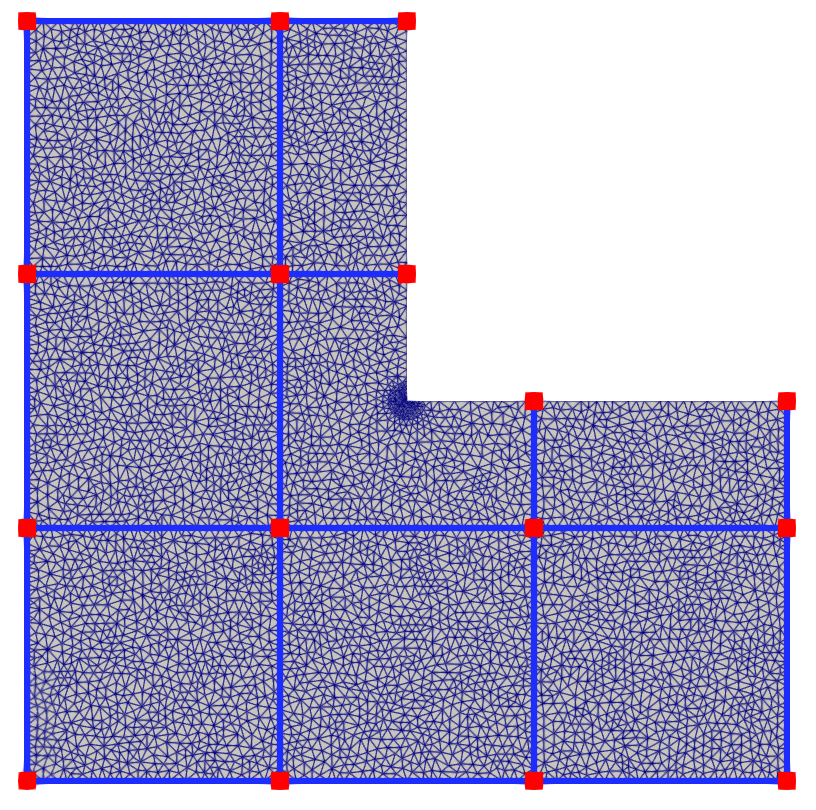}
		\end{subfigure}
		\begin{subfigure}{0.49\textwidth}
			\centering
			\includegraphics[height=4.cm, width=4.5cm]{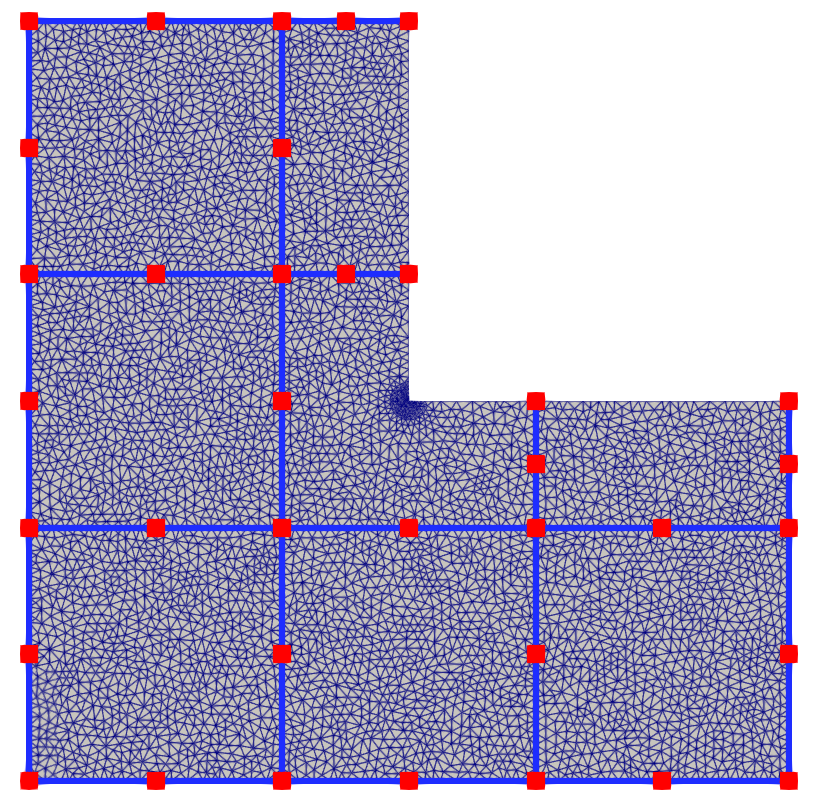}
		\end{subfigure}
		\caption{
			Coarse and fine discretizations of the L-shaped domain with 0 (left) and 1 (right) additional degree(s) of edge refinement.
		}
		\label{fig:edgeref}
	\end{figure}

	In order to assess the convergence of the discrete Trefftz method, we consider two  strategies regarding the refinement of the coarse partitioning. The procedure involving the reduction of the diameter of the coarse cells will be referred to as \emph{mesh refinement}. 
	The sequence of such meshes will be constructed as follows: first the background domain $D$
	is partitioned into $N=(2k  + 1)^2$,  $k \in \mathbb{N}$,  squares,  then the coarse cells $\O_j$ are generated by excluding $\O_S$. 
	The choice of $N$ being a square of an odd number ensures the consistency of the mesh sequence in terms of the shape of the elements. The second considered refinement strategy will be referred  to as \emph{edge refinement} procedure,  which involves subdividing the edges of an original ``$3\times 3$ grid''.  This edge refinement approach is illustrated by Figure \ref{fig:edgeref} and is inspired by the Multiscale Hybrid-Mixed method \cite{mhm}. We remark that with both refinement procedures, it is ensured that none of the coarse grids will have a degree of freedom located at the corner $(0,0)$.  As a result,  the corner singularity will be captured by the basis functions associated with the L-shaped domain.  We  also note that in order to improve the precision of the fine-scale finite element method, the size of the triangles is graded in the vicinity of the corner $(0,0)$.  
	
	Figure \ref{fig:approxerror} reports, for both mesh and edge refinement strategies, the relative error in $H^1$ semi-norm and $L^2$ norm as functions of maximal coarse edge length $H$. The black dashed line represents the typical fine-scale finite element error; here, we use $\mathbb{P}_2$ finite elements.  The relative error and the experimental order of convergence for edge refinement procedure are equally reported in Table \ref{tab_L-shaped}.   
 In accordance with the error estimate \eqref{H1_est},  for the edge refinement procedure, we observe superconvergence in the energy norm with rates of approximately $3/2$ and $5/2$ for  $p = 1$ and $p = 2$, respectively. In the $L^2$ norm, the observed convergence rates are approximately $3$ and $7/2$ for $p = 1$ and $p = 2$, respectively. 
 In contrast to the edge refinement convergence, the convergence resulting from the mesh refinement seems to be controlled by the low global regularity of the solution.  We observe the convergence rates typical for finite element methods on quasi-uniform meshes, that is, close to $2/3$ and $4/3$ in $H^1$ and $L^2$ norms, respectively. 

\begin{table}
\begin{center}
\scalebox{0.9}
{
\begin{tabular}{|c|c|c|c|c|c|c|c|c|c|c|c|c|c|c|c|c|c|c|c|}
\hline
& \multicolumn{4}{|c}{$H^1$ semi-norm} & \multicolumn{4}{|c|}{$L^2$ norm}  \\ \hline
Ref.  lvl. & $p = 1$ & eoc & $p = 2$ & eoc  & $p = 1$ & eoc & $p = 2$ & eoc   \\ \hline
0 & -1.098 & - & -1.812 & - & -1.881 & - & -2.864 & -  \\ \hline
1 &-1.601 & 1.669 & -2.665 & 2.833 & -2.762 & 2.925 & -4.026 & 3.858  \\ \hline
2 &-2.132 & 1.765 & -3.35 & 2.274 & -3.677 & 3.042 & -4.993 & 3.215  \\ \hline
3 &-2.626 & 1.641 & -4.043 & 2.305 & -4.514 & 2.781 & -6.014 & 3.389  \\ \hline
\end{tabular}
}
\caption{Relative error and the experimental order of convergence for the Trefftz approximation with $p = 1, 2$ and using the edge refinement. Ref. level refers to the degree of additional edge refinement. Eoc refers to error of convergence.}
\label{tab_L-shaped}
\end{center}
\end{table}

	\begin{figure}
		\centering
		\begin{subfigure}{0.49\textwidth}
			\centering
			\includegraphics[width=\linewidth, height=4.5cm]{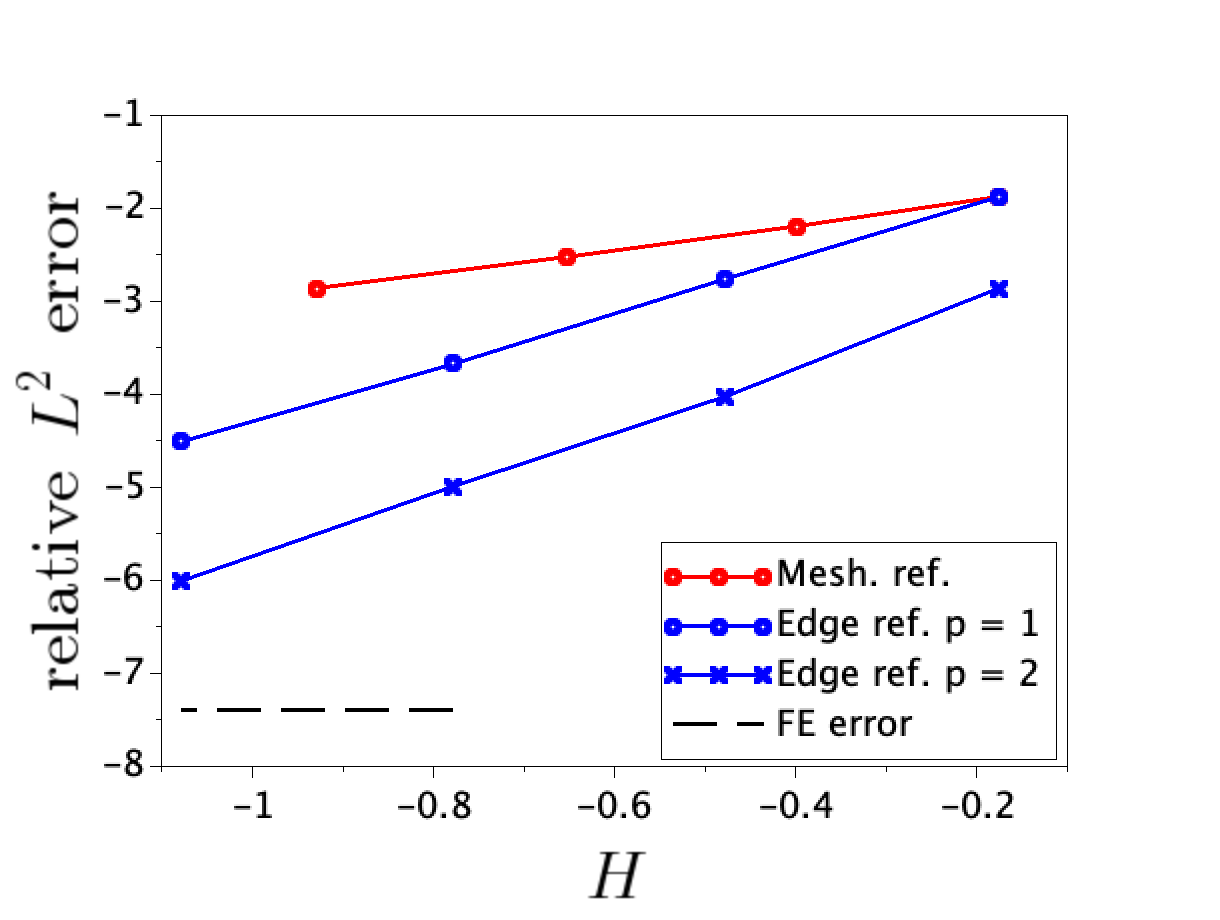}
  \end{subfigure}
		\begin{subfigure}{0.49\textwidth}
			\centering
			\includegraphics[width=\linewidth, height=4.5cm]{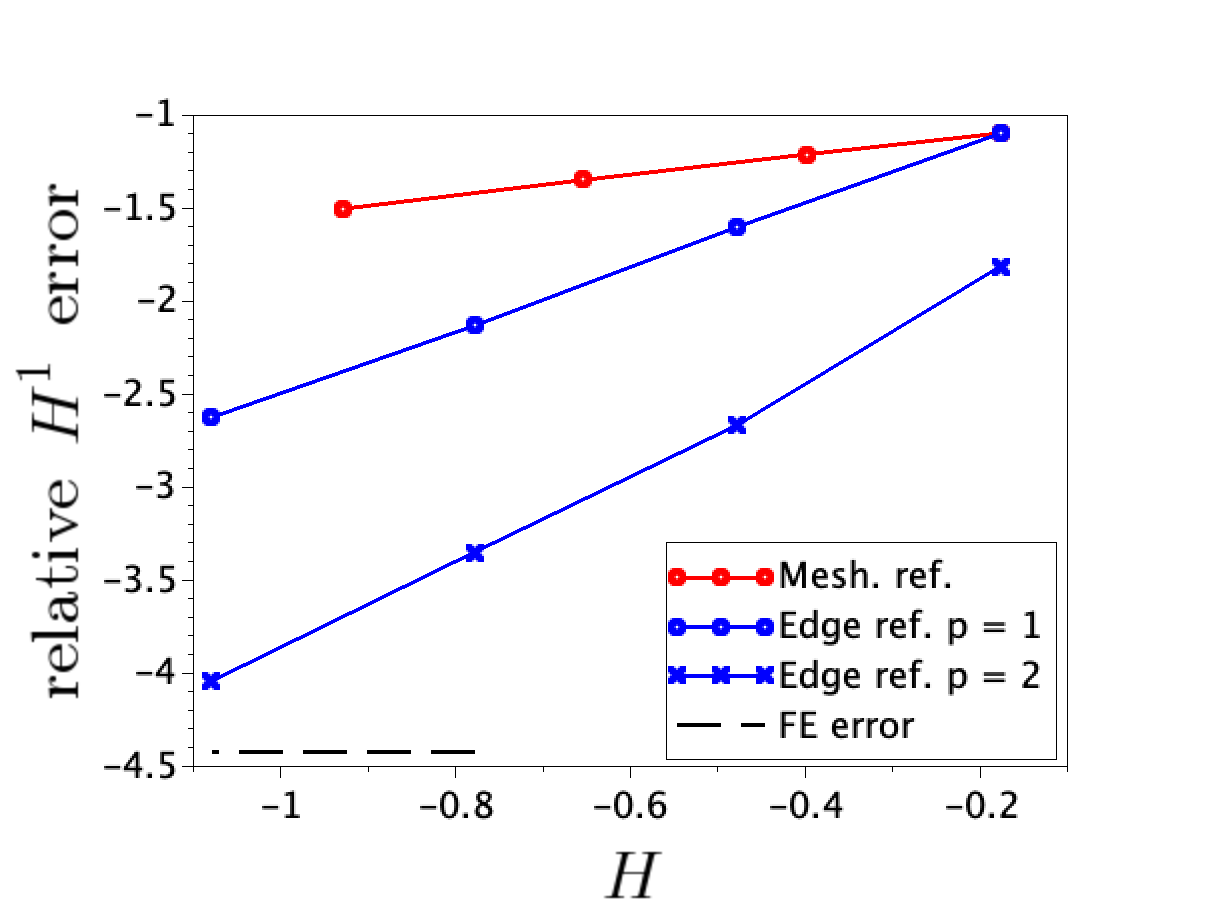}
        \end{subfigure}
		\caption{
			Coarse approximation error for L-shaped domain with edge (blue) and mesh (red) refinement in $L^2$ norm (left) and the energy norm (right). The black dashed line denotes the finite element error.
	}
	\label{fig:approxerror}
\end{figure}


\subsection{Iterative DD and Preconditioned Krylov methods on L-Shaped Domain}

We report on Figure \ref{fig:iterativelshapedoverlap} the convergence history of the iterative methods for the L-shaped domain. 
More precisely, we report two different error metrics, referred to as algebraic error and full error. 
For a given iterate, algebraic error is defined as a distance (based on $L^2$ or $H^1$ norms) between a given iterate and the solution of the algebraic system \eqref{Au_f}, while full error reflects the distance to the exact solution of \eqref{model_pde}.
Figure \ref{fig:iterativelshapedoverlap} reports the convergence histories for linear systems 
resulting from an increasingly accurate background finite element discretization. 
\blue{The overlap is set to $\smallop$, where $\mathcal{H}_j= \max (x_{max,j}-x_{min,j}, y_{max,j}-y_{min,j})$ and $x_{min,j}, y_{min,j}, x_{max,j}, y_{max,j}$ denote the minimal and maximal $x$ and $y$ coordinates that are contained in $\O_j$.} 
We consider Trefftz spaces of order $p = 1$ and $p = 2$.


Because the width of the overlap is maintained constant, we observe on Figure \ref{fig:iterativelshapedoverlap} that the convergence the iterative methods is not sensitive to the accuracy of the background finite element method until the precision of the latter is reached. The convergence of preconditioned GMRES method is exponential, while the convergence curve of the algebraic error of the two-level fixed point method \eqref{eq:iterativeras} exhibits two distinct slopes, with fast convergence at the initial stage. We note that the performance of both the stationary iteration and the preconditioned GMRES algorithm is improved by the higher order Trefftz space.
We note that the initial slope of the stationary iteration method is approaching the slope of the preconditioned GMRES as the Trefftz order increases.
In general, the results obtained in terms of the full error (right column of Figure \ref{fig:iterativelshapedoverlap}) seem to point toward the following conclusion: the stationary iteration method appears to be an acceptable alternative to the preconditioned GMRES if very high accuracy is not required.
\begin{figure}
	\centering
	\begin{subfigure}{0.49\textwidth}
		\centering
		\includegraphics[width=\linewidth, height=4cm]{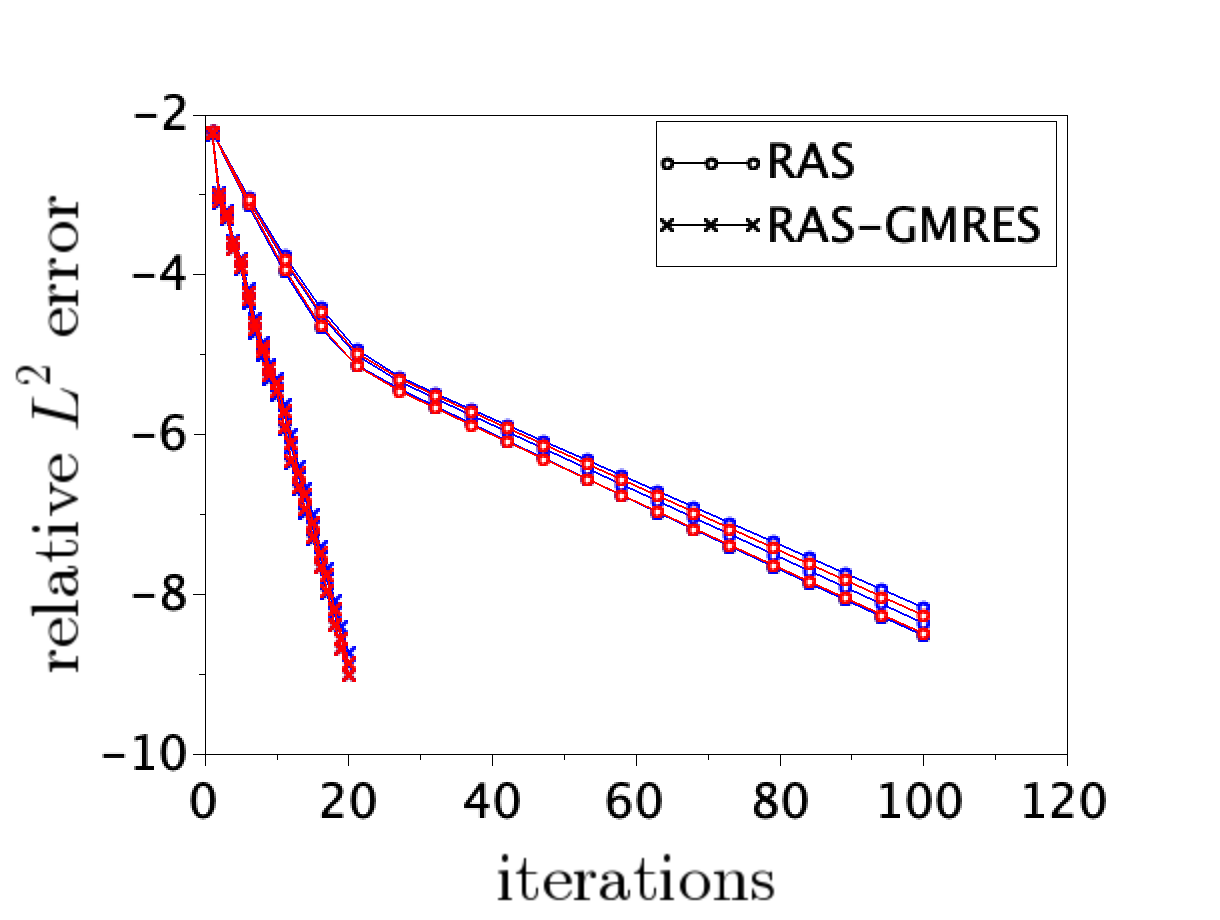}
	\end{subfigure}
	\begin{subfigure}{0.49\textwidth}
		\centering
		\includegraphics[width=\linewidth,height=4cm ]{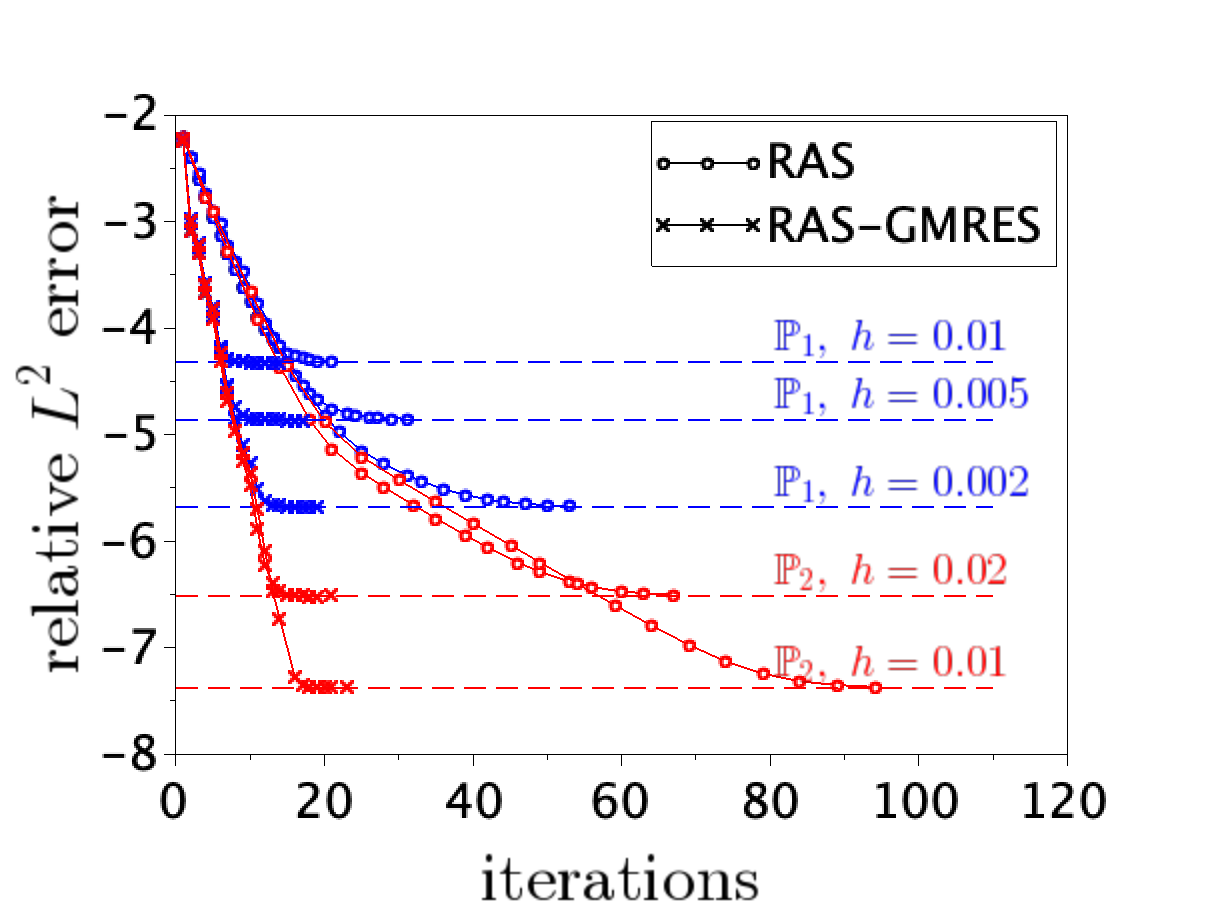}
	\end{subfigure}
 	\begin{subfigure}{0.49\textwidth}
		\centering
		\includegraphics[width=\linewidth, height=4cm]{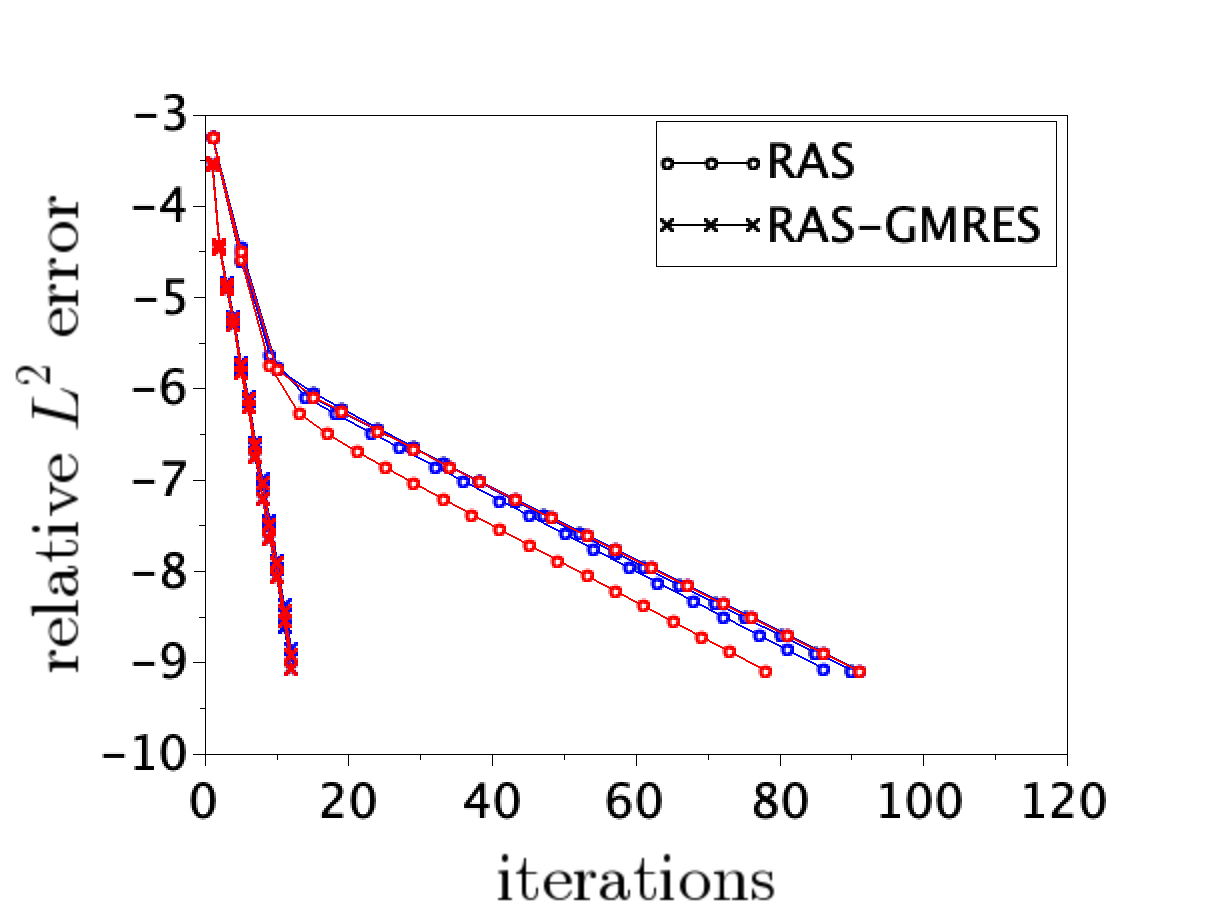}
	\end{subfigure}
	\begin{subfigure}{0.49\textwidth}
		\centering
		\includegraphics[width=\linewidth,height=4cm ]{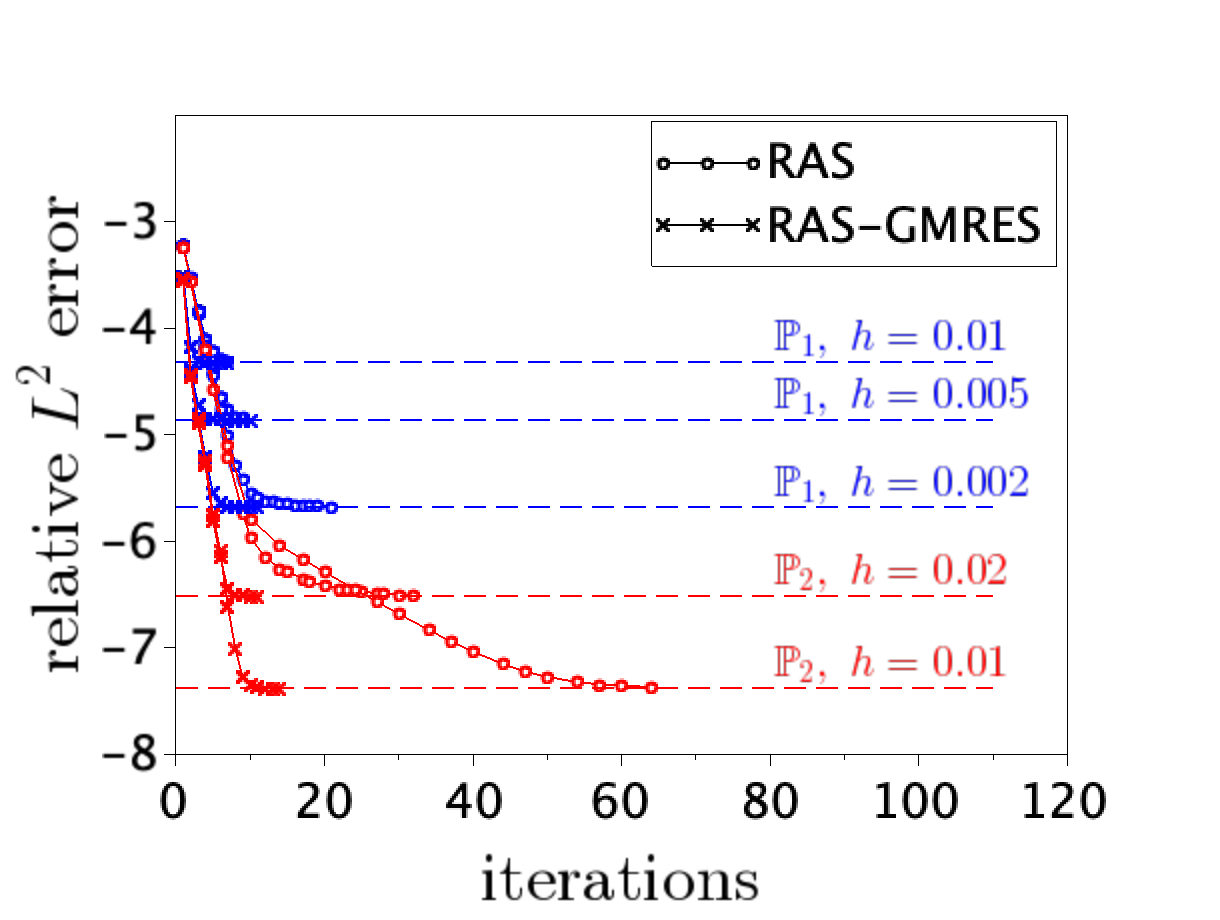}
	\end{subfigure}
	\caption{Convergence in $L^2$ of the two-level iterative (RAS) and preconditioned GMRES (RAS-GMRES) methods for the L-shaped domain on $5 \times 5$ subdomains with overlap of $\smallop$ and using first (top row) and second (bottom row) order Trefftz approximation.
		The dashed horizontal lines show the error of the fine-scale finite element method. Left to right: algebraic error,  full error. Results for finite elements of order 1 and 2 are shown in blue and red, respectively.
	}
	\label{fig:iterativelshapedoverlap}  
\end{figure}


We further study the impact of edge refinement on the performance of the iterative methods.
We report on Figure \ref{fig:iterativeLedgeref} 
the convergence of algebraic error in $H^1$ and $L^2$ norms for edge refinement up to degree three.
This is reported for both the iterative and preconditioner approaches.
We observe that reducing $H$ not only improves the initial coarse approximation (the first iteration point), but also accelerates the convergence of the iterative methods. The most notable is the impact on the initial slope of the two-level iterative RAS method.
Here, again, the initial slope of the stationary iteration method comes close to that of the preconditioned GMRES as we increase the order of edge refinement.

\begin{figure}
	\centering
	\begin{subfigure}{0.49\textwidth}
		\centering
		\includegraphics[width=.9\linewidth, height=4.5cm]{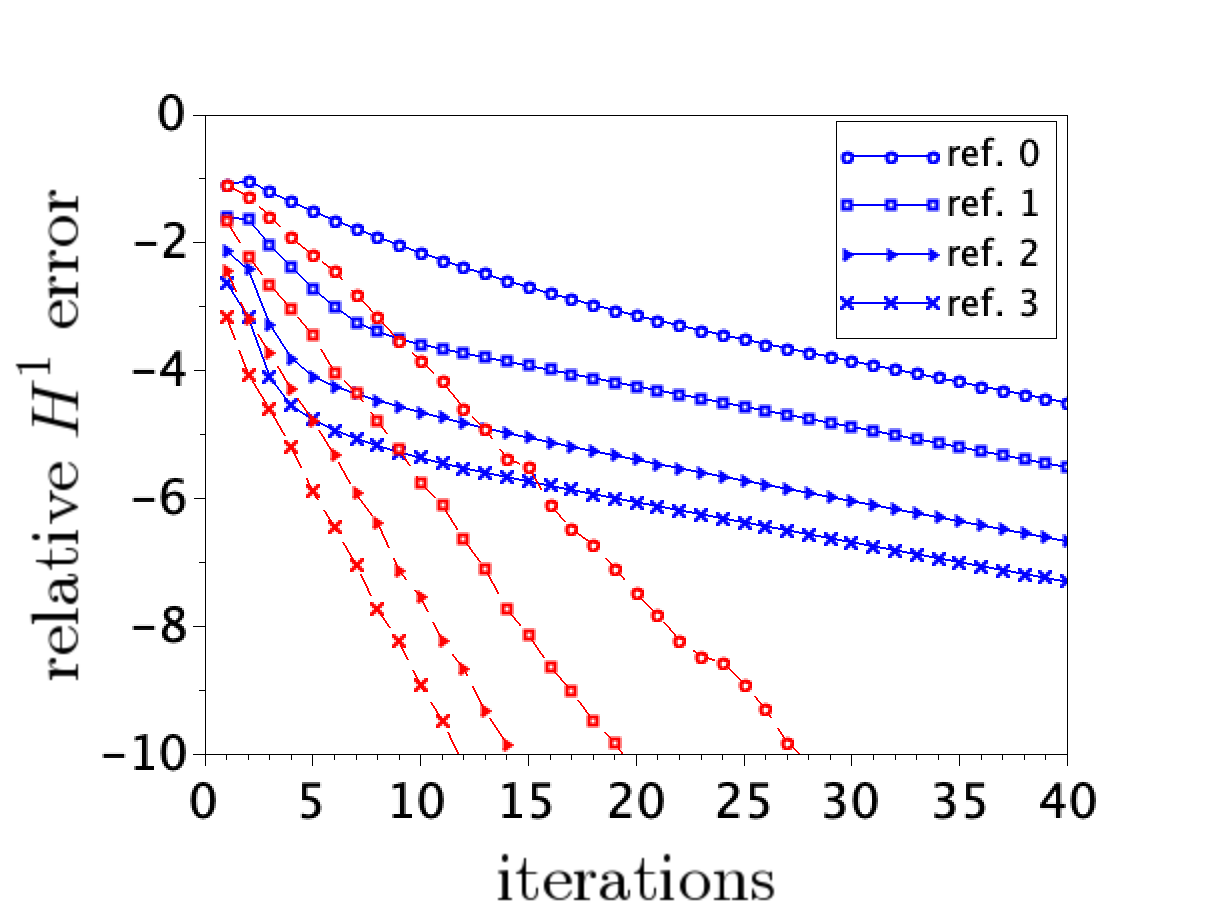}
	\end{subfigure}
	\begin{subfigure}{0.49\textwidth}
		\centering
		\includegraphics[width=.9\linewidth, height=4.5cm]{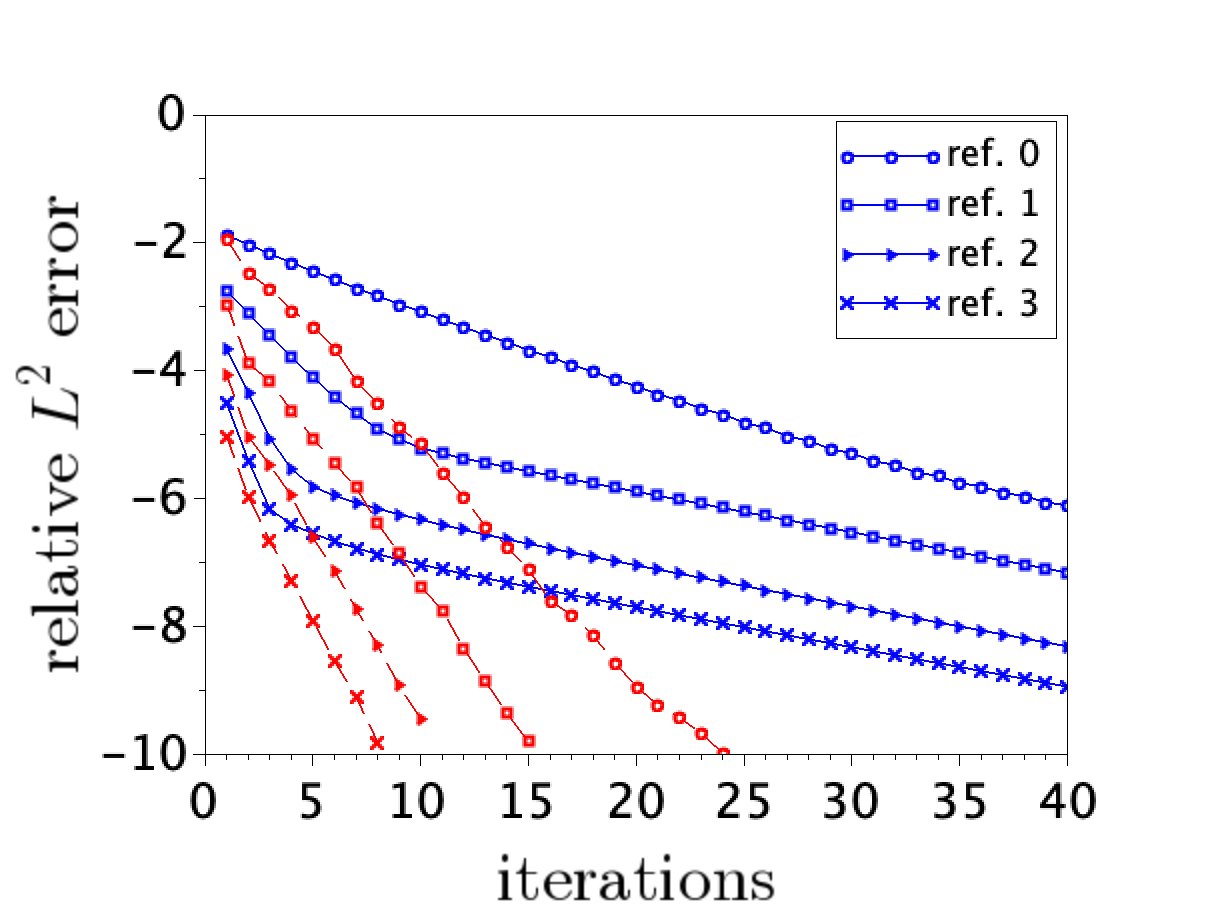}
	\end{subfigure}
	\caption{
        Convergence in $H^1$ (left) and $L^2$ (right) norms of the two-level iterative (blue) and preconditioned GMRES (red) methods for the L-shaped domain on $3 \times 3$ subdomains with overlap of $\smallop$.
        Algebraic errors are shown for the first order Trefftz space with various degrees of edge refinement.
}
\label{fig:iterativeLedgeref}
\end{figure}

\subsection{Iterative DD and Preconditioned Krylov Methods on Small Urban Domain} 
We examine the performance of the two-level stationary iteration and preconditioned GMRES method over a domain based on realistic urban geometries for which the data sets were kindly provided by Métropole Nice Côte d'Azur. 
We focus here on a relatively small spatial frame shown on the left of Figure \ref{fig:framesols}). 
The data frame is $160 \times 160$ meters and contains two kinds of structural features representing buildings (and assimilated small elevated structures) and walls in urban data. The number of perforations associated to buildings and walls is 63 and 77, respectively.
The triangulation of the domain is performed using Triangle \cite{Triangle}. 
\blue{This ``small" model domain would take a minimum of 23\, 000 degrees of freedom to triangulate without imposing maximum triangle area or constraining the mesh to match the coarse partitioning. While the fine-scale mesh will change depending on $N$, the number of degrees of freedom in the mesh will be close to this number.} Due to the geometric complexity of the computational domain, the mesh includes a number of very small triangles; the minimal triangle area is given as $5.83\times 10^{-13}$.
We consider the model problem \eqref{model_pde} with $f = 1$, which we discretize with continuous piece-wise affine finite elements; the finite element solution is reported on the left of Figure \ref{fig:framesols}.

We report on Figures \ref{fig:iterativeL80withref} and \ref{fig:iterativeL80withrefratio0} the convergence history of the stationary iterative method for different sizes of overlap.
This figure shows results for both algebraic and full errors for various fine-scale discretization sizes $h$.
The performance of the preconditioned GMRES is reported on Figures \ref{fig:L80kryratio5} and \ref{fig:L80kryratio0}.
In lieu of a true solution (which would exhibit multiple singularities), we provide a fine grid numerical solution to which the iterates are compared. The grid for the reference solution is generated via multiple levels of edge bisection (``red refinement''). We note that in contrast to the previous numerical experiment, we do not perform any mesh grading near the reentering corners of the domain; as a result, the finite element error is high.

We expose on Figure \ref{fig:iterativeL80withref} the convergence history of the two-level stationary iteration method with overlap $\smallop$. We observe that the convergence of the algebraic error is robust with respect to $h$ and that the fine-scale finite element method (black lines) can be reached in a few steps. 
Further iterations do not improve the overall precision of the approximate solution even though the algebraic error may decrease.
A stopping criterion which prevents excess iterations is warranted; such a posterioiri estimates are discussed for example in \cite{stop, vohl}. 
As expected, the performance of the stationary iteration method considerably deteriorates (see Figure \ref{fig:iterativeL80withrefratio0}) in the case of minimal geometric overlap.


\begin{figure}
\begin{subfigure}{0.49\textwidth}
	\centering
	\includegraphics[width=\linewidth, height=4cm]{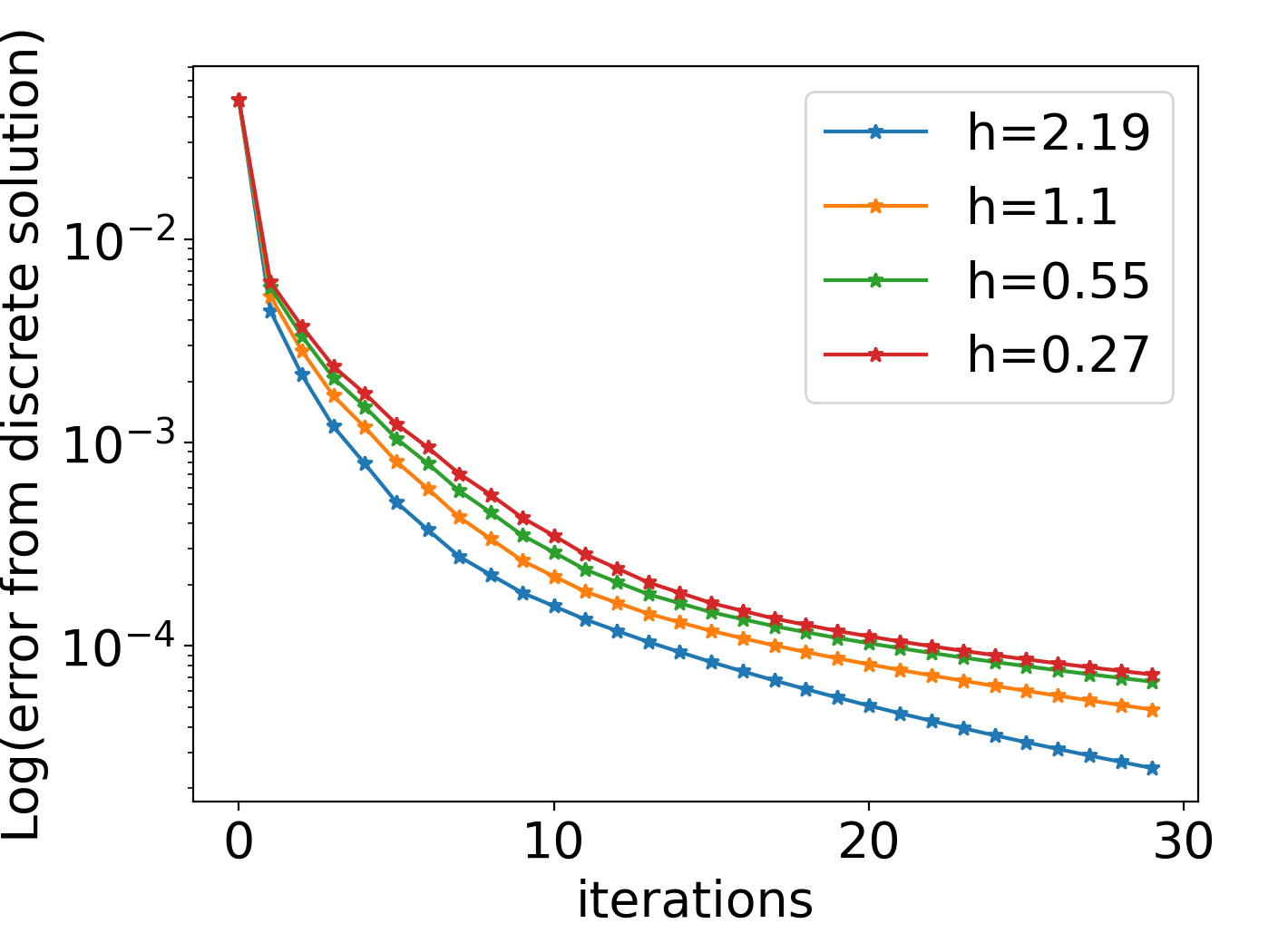}
\end{subfigure}
\begin{subfigure}{0.49\textwidth}
	\centering
	\includegraphics[width=\linewidth, height=4cm]{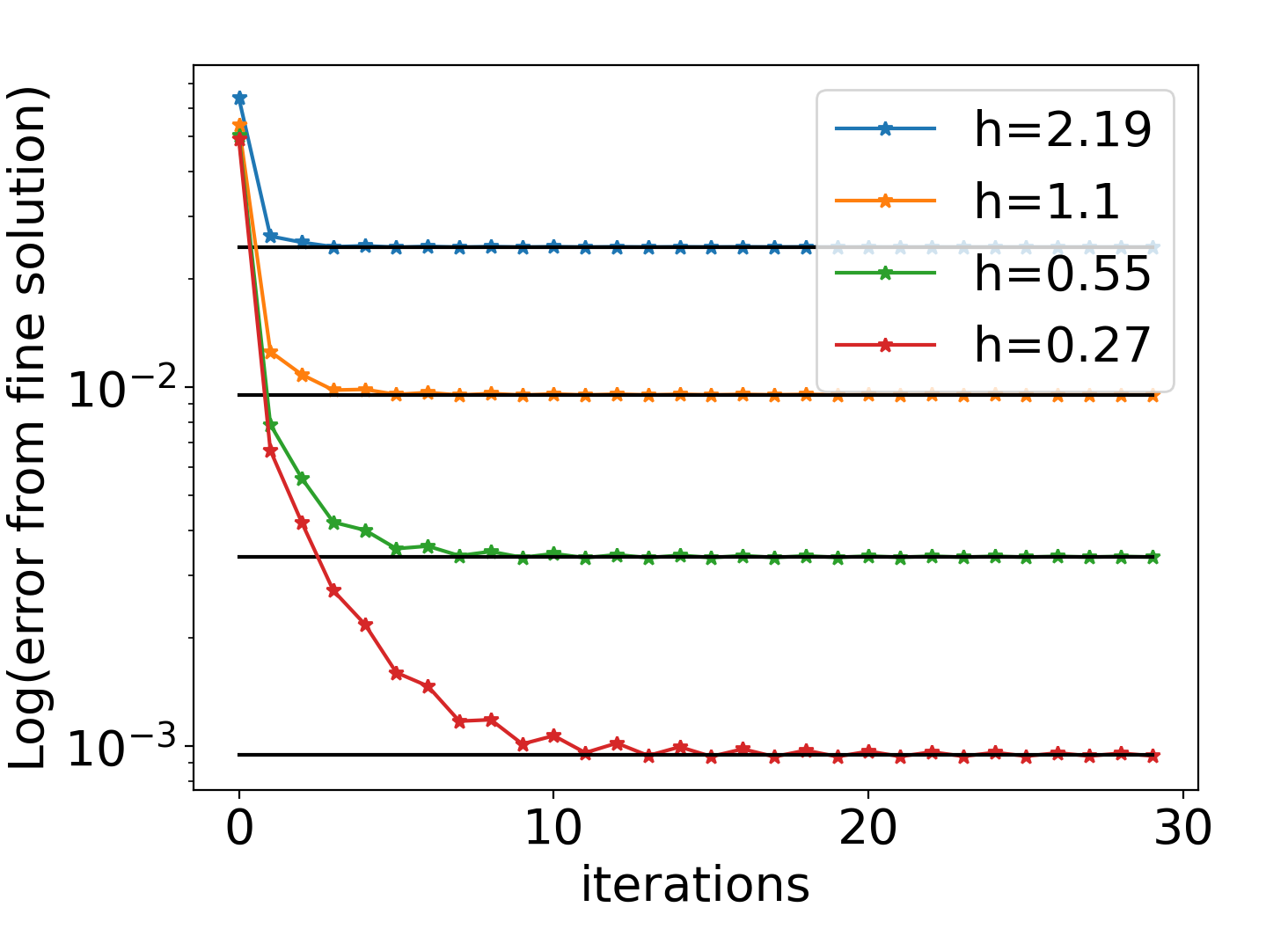}
\end{subfigure}
\caption{Convergence of the two-level iterative method for the urban data set on $8 \times 8$ subdomains with overlap $\smallop$.
	The black horizontal lines show the error of the fine-scale finite element method.  Left to right: algebraic error, full error.
}
\label{fig:iterativeL80withref}
\end{figure}

\begin{figure}
\begin{subfigure}{0.49\textwidth}
	\centering
	\includegraphics[width=\linewidth, height=4cm]{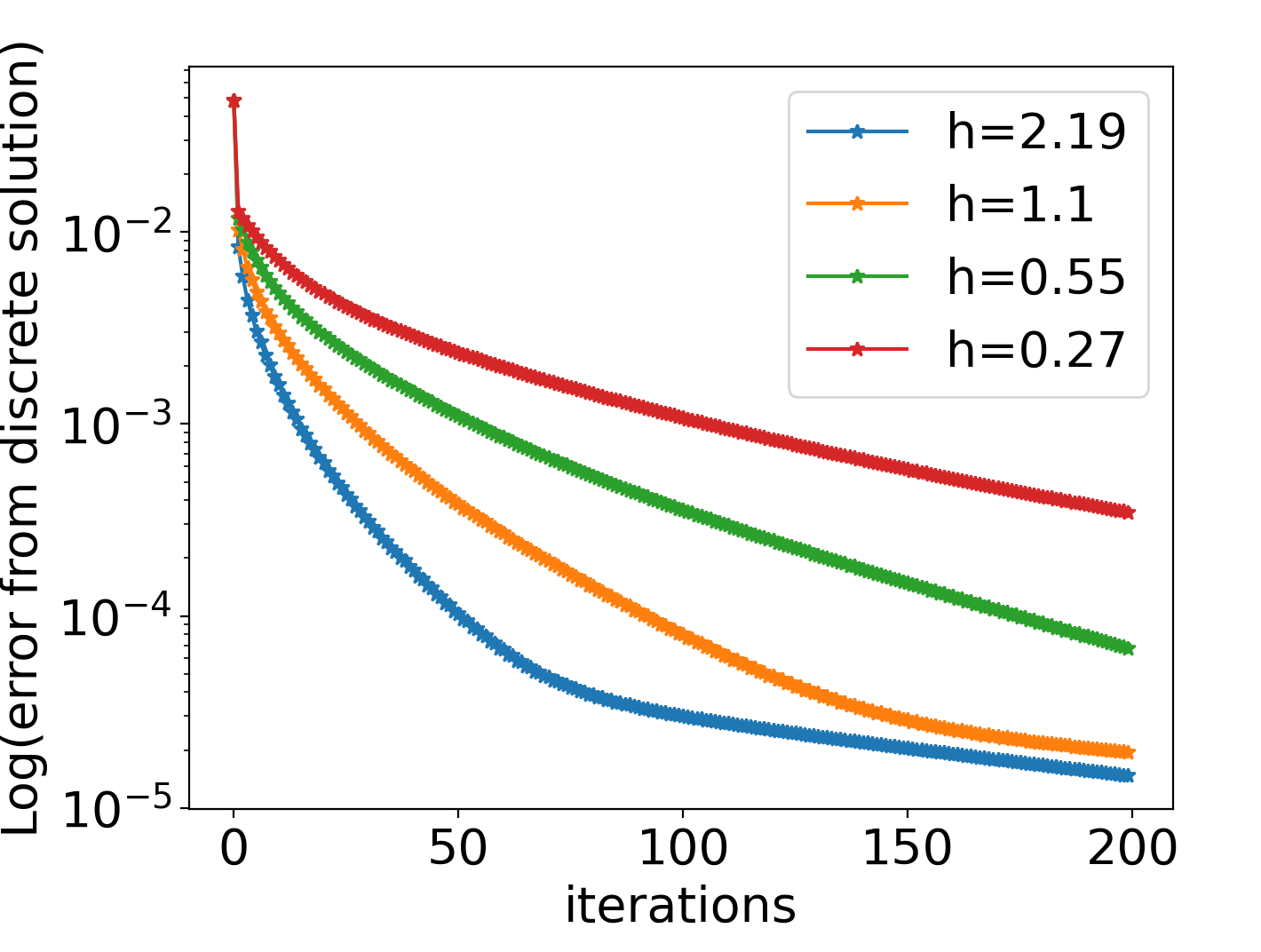}
\end{subfigure}
\begin{subfigure}{0.49\textwidth}
	\centering
	\includegraphics[width=\linewidth, height=4cm]{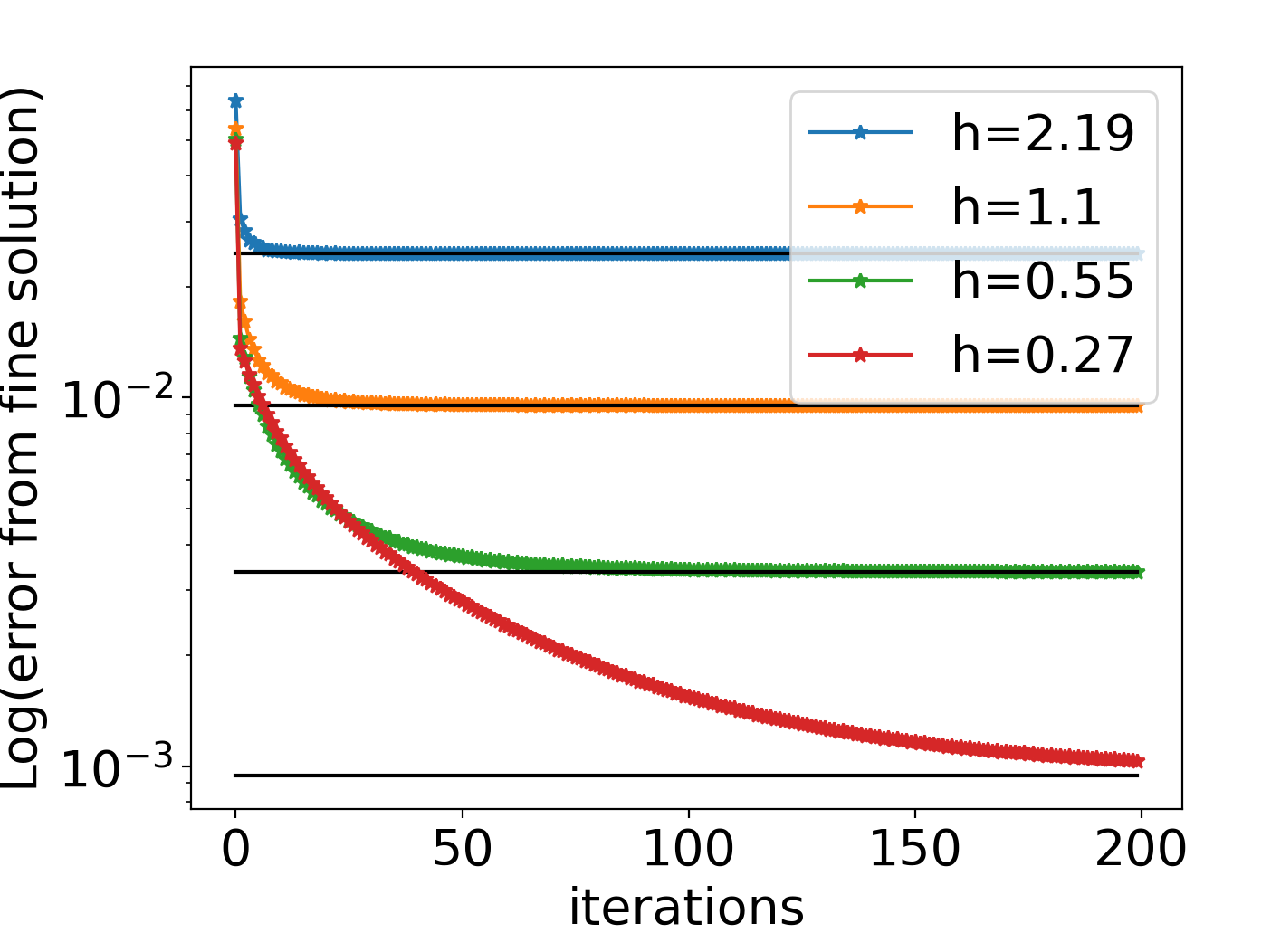}
\end{subfigure}
\caption{Convergence of the two-level iterative method for the urban data set on $8 \times 8$ subdomains with minimal geometric overlap.
	The black horizontal lines show the error of the fine-scale finite element method.  Left to right: algebraic error, full error.
}
\label{fig:iterativeL80withrefratio0}
\end{figure}


Figures \ref{fig:L80kryratio5} and \ref{fig:L80kryratio0} report the performance of the two-level preconditioned GMRES method 
for overlaps of $\smallop$ and minimal geometric overlap, respectively. Compared to the stationary iterative method, a major improvement is achieved for the case of minimal geometric overlap. However, due to the low accuracy of the background finite element discretization, both the iterative and preconditioned GMRES methods perform comparably for the overlap of $\smallop$.



\begin{figure}
\begin{subfigure}{0.49\textwidth}
	\centering
	\includegraphics[width=\linewidth, height=4cm]{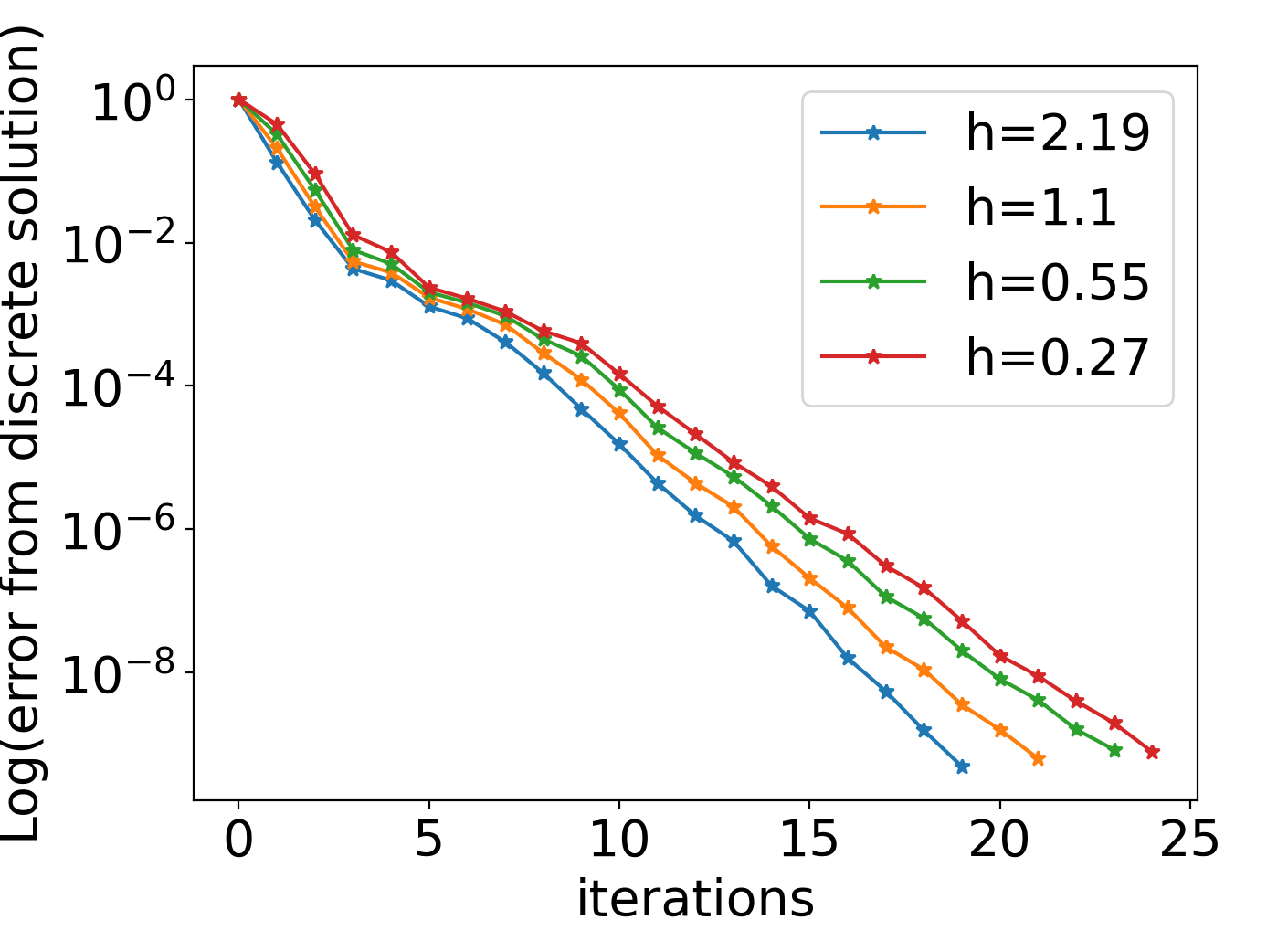}
\end{subfigure}
\begin{subfigure}{0.49\textwidth}
	\centering
	\includegraphics[width=\linewidth, height=4cm]{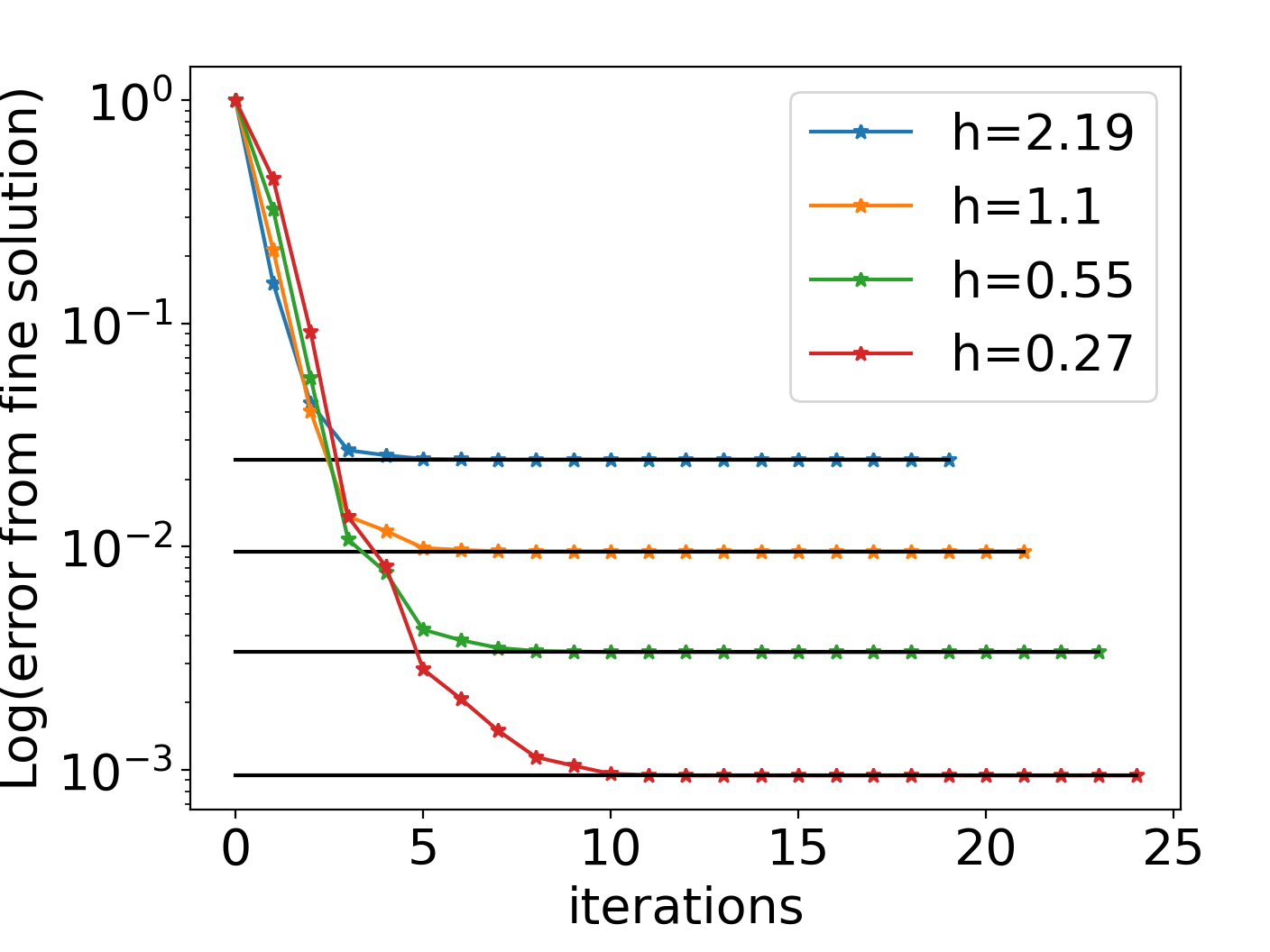}
\end{subfigure}
\caption{Convergence of the two-level preconditioned Krylov method for the urban data set on $8 \times 8$ subdomains with overlap $\smallop$.
	The black horizontal lines show the error of the fine-scale finite element method.  Left to right: algebraic error, full error.
}
\label{fig:L80kryratio5}
\end{figure}

\begin{figure}
\begin{subfigure}{0.49\textwidth}
	\centering
	\includegraphics[width=\linewidth, height=4cm]{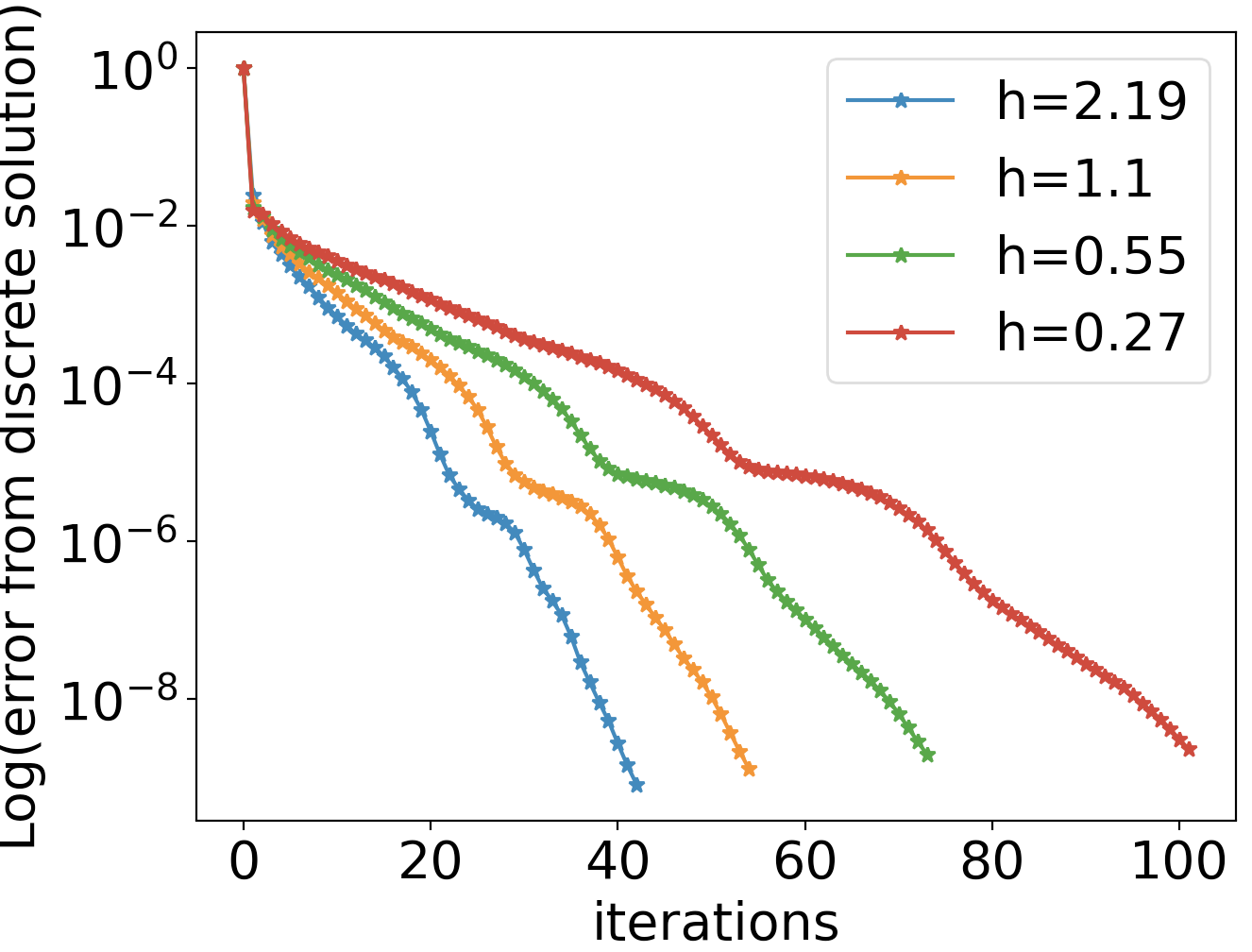}
\end{subfigure}
\begin{subfigure}{0.49\textwidth}
	\centering
	\includegraphics[width=\linewidth, height=4cm]{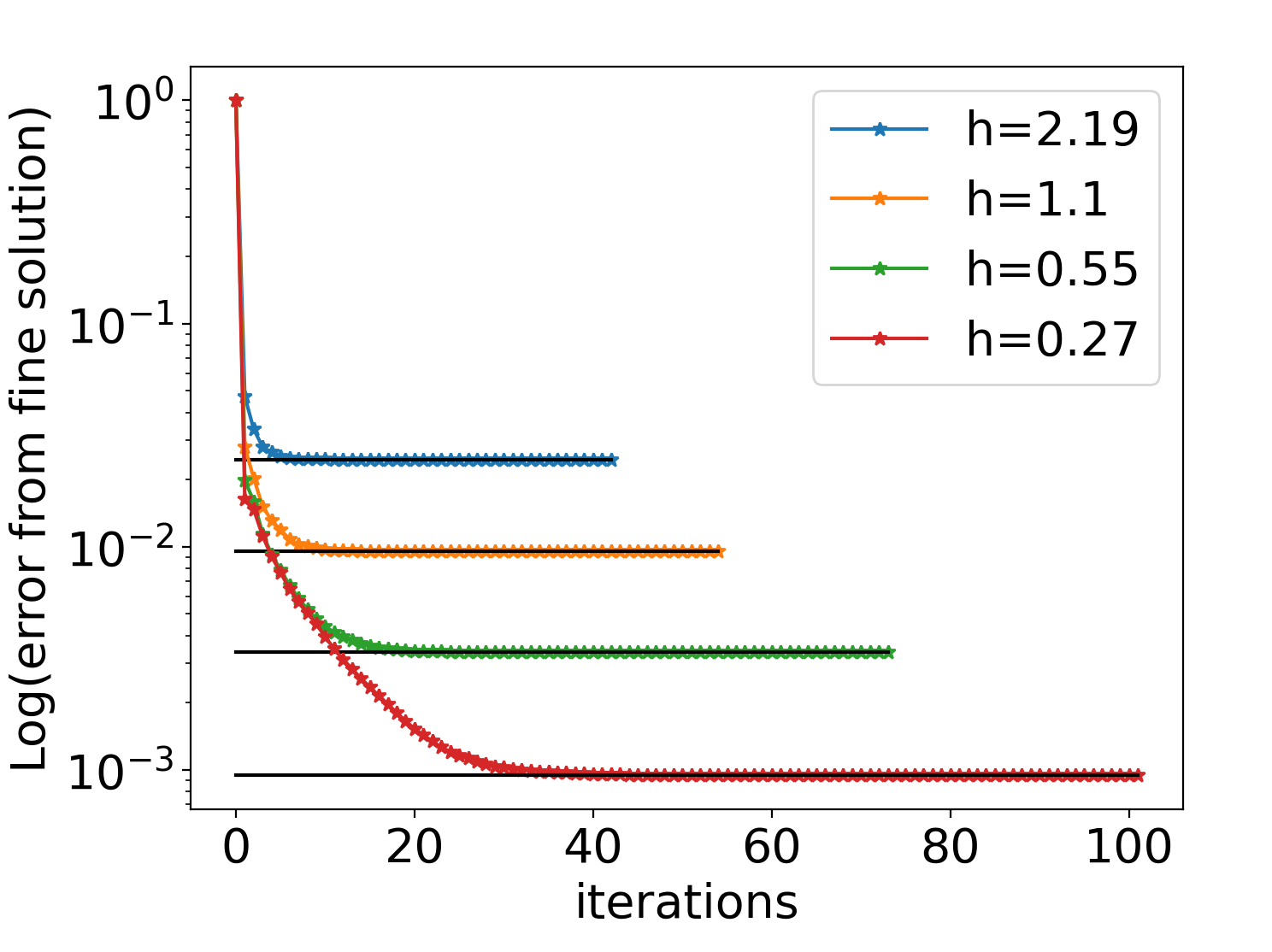}
\end{subfigure}
\caption{Convergence of the two-level preconditioned Krylov method for the urban data set on $8 \times 8$ subdomains with minimal geometric overlap.
	The black horizontal lines show the error of the fine-scale finite element method.  Left to right: algebraic error, full error.
}
\label{fig:L80kryratio0}
\end{figure}

\subsection{Scalability Tests for Preconditioned Krylov Method on Larger Urban Domain}

As a final numerical experiment, we follow the experiment from the previous subsection and report the performance of the two-level preconditioned GMRES method on a larger data set shown on the right of Figure \ref{fig:framesols}. 
This larger data set contains 306 buildings and 477 walls of varying sizes, the dimensions of the domain are $640\times 640$ metres.
\blue{The ``large" model domain would take a minimum of 142\,700 degrees of freedom to triangulate without imposing maximum triangle area or constraining the mesh to match the coarse partitioning.} The minimal triangle area is observed to be $5.83\times 10^{-13}$.

We note that, depending on the partitioning of the domain, the perforations resulting from this data set (especially the wall data) can span across multiple coarse cells,  which is a challenging situation for traditional coarse spaces.
To further explore this observation, we will provide numerical results for data sets that do and do not include walls as perforations. Figure \ref{fig:a} reports the finite element solution obtained for a data set excluding and including walls. We can see from Figure \ref{fig:a} that the inclusion of walls noticeably affects the numerical solution.

\begin{figure}
\centering
\begin{subfigure}{0.49\textwidth}
	\centering
	\includegraphics[width=.9\linewidth]{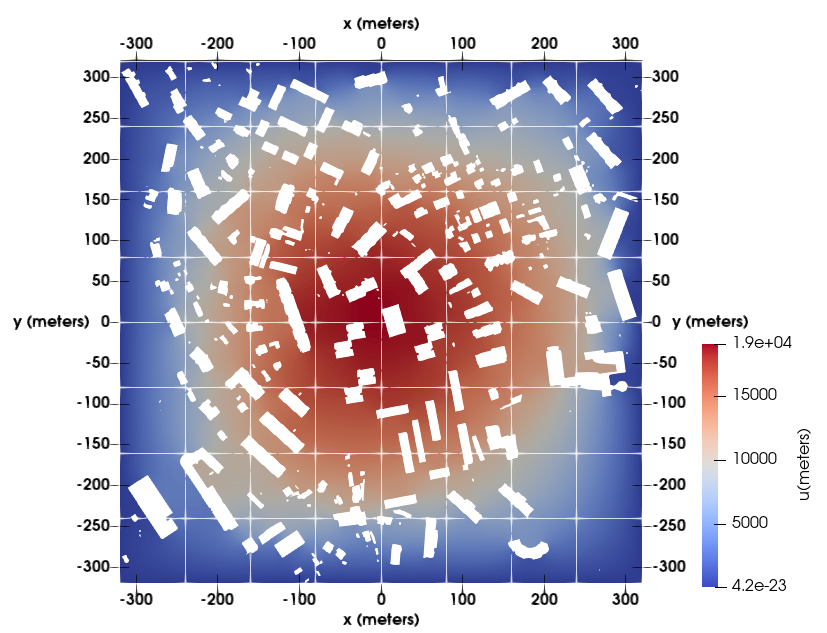}
	\caption{Without walls}
\end{subfigure}
\begin{subfigure}{0.49\textwidth}
	\centering
	\includegraphics[width=.9\linewidth]{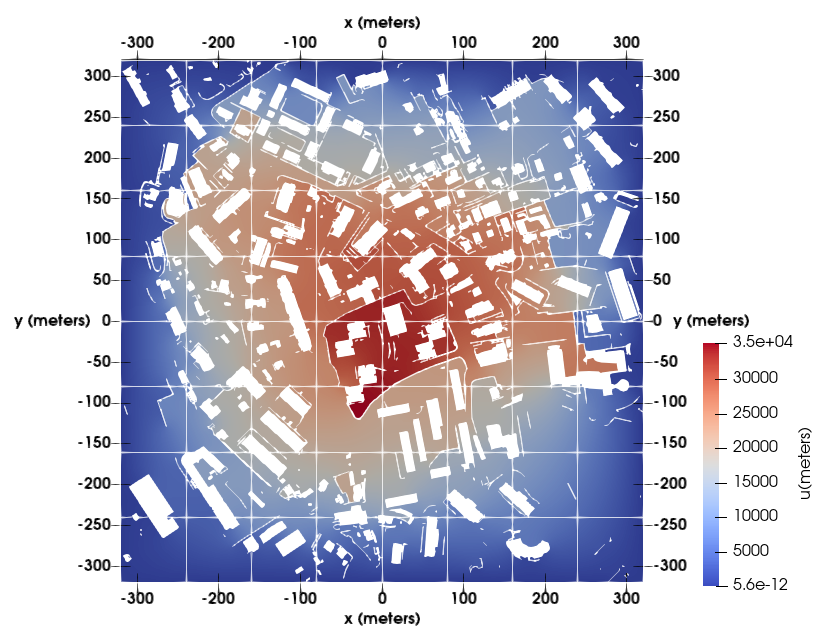}
	\caption{With walls}
\end{subfigure}
\caption{
	Approximate finite element solution over a computational domain divided in $N=8 \times 8$ nonoverlapping subdomains.
}
\label{fig:a}
\end{figure}

For the sake of comparison, we also provide numerical results for the well-known Nicolaides coarse space \cite{nic}, made of flat-top partition of unity functions associated with the overlapping partitioning.  
For the Nicolaides coarse space, scalability with respect to the number of subdomains is only guaranteed if the subdomains are connected.  Therefore, to generate connected subdomains, we further partition $(\O_j')_j$ into a family of connected subdomains. 
In other words, let $m_j$ denote the number of disconnected components for each overlapping subdomain $\O_j'$ and let $\O_{j,l}', l=1, \ldots, m_j$ denote the corresponding disconnected component. Then our new overlapping partitioning contains $m=\sum_{j=1}^{N} m_j$ total subdomains and is given by
$$ (\widehat{\O}_{k})_{k \in \{1, \ldots, m\}} = \bigl( (\O_{j,l}')_{l \in \{1, \ldots, m_j\}} \bigr)_{j \in \{1, \ldots, N\}}. $$
This special partitioning in visualized in Figure \ref{fig:nic}.

\begin{figure}
\centering
\begin{subfigure}{0.495\textwidth}
\centering
\includegraphics[width=.7\linewidth, height=.7\linewidth]{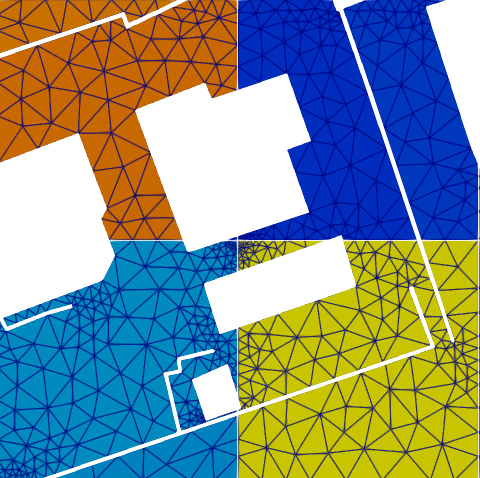}
\end{subfigure}
\begin{subfigure}{0.495\textwidth}
\centering
\includegraphics[width=.7\linewidth, height=.7\linewidth]{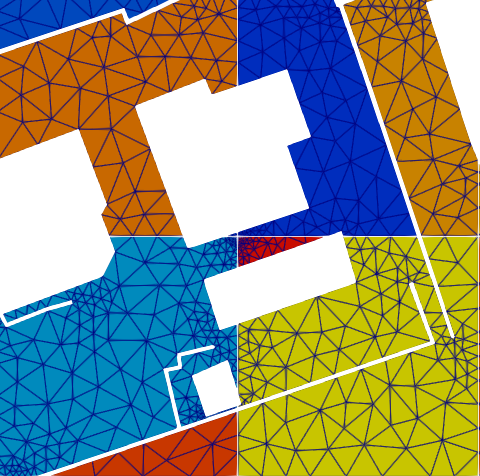}
\end{subfigure}
\caption{Possible partitioning of the domain without (left) and with (right) additional partitioning by disconnected component, with each subdomain/component depicted by a color block.}
\label{fig:nic}
\end{figure}

Then, the Nicolaides coarse space is as follows. The $k$th row of the discrete coarse matrix $\bR_H$ and therefore the $k$th column of $\bR_H^T$ is given by $$(\bR_H^T)_k=  \widehat{\bR}_k^T \widehat{\bD}_k \widehat{\bR}_k\mathbf{1},$$
for $k=1, \ldots, m$, where $\widehat{\bR}_k$ and $\widehat{\bD}_k$ are the restriction and partition of unity matrices corresponding to $\widehat{\O}_{k}$ and $\mathbf{1}$ is a vector full of ones.

Figure \ref{fig:convcurves} reports 
convergence histories of the preconditioned GMRES method using the Nicolaides and Trefftz coarse spaces for the data set including both walls and buildings.
Table \ref{tab:table} summarizes the numerical performance for data sets including or excluding walls.  In particular, for both preconditioners,  it reports 
the dimensions of the coarse spaces, as well as the number of GMRES iterations required to achieve a relative $L^2$ error of $10^{-8}$.

 As we provide numerical results for various numbers of subdomains $N$ and the computational domain $\O$ remains fixed independently of $N$, the results of this experiment should be interpreted in terms of a strong scalability.  However, we wish to stress that the fine-scale triangulation is conforming to the nonoverlapping partitioning $(\O_j)_{j = 1,\ldots, N}$.  Consequentially, the system \eqref{Au_f} changes from one coarse partitioning to another.  Nevertheless we ensure that the dimension of  the system \eqref{Au_f} is roughly constant throughout the experiment for a given $N$.  

The performance of the Trefftz coarse space appears to be very robust with respect to both $N$ and the complexity of the computational domain.  While the Nicolaides space is also fairly robust with respect to $N$ (as expected), the Trefftz space provides an additional acceleration in terms of iteration count. The improvement with respect to the alternative Nicolaides approach is quite striking, particularly in the case of the minimal geometric overlap.
As expected,  increased overlap in the first level of the Schwarz preconditioner provides additional acceleration in terms of iteration count.  However,  for the Trefftz space,  the results with minimal geometric overlap appear to already be quite reasonable. 

The dimensions and the relative dimensions of the two coarse spaces are reported in Table \ref{tab:table}.
Relative dimension refers to the would-be dimension of the coarse space in the case of a domain without perforations
with $\O_S = \emptyset$;  that is, the relative dimensions are computed as $\frac{\text{dim}(\bR_H)}{(\sqrt{N}+1)^2}$ for the Trefftz space and as $\frac{\text{dim}(\bR_H)}{N}$ for the Nicolaides space.  
We observe that the Trefftz space requires a much larger number of degrees of freedom,
which naturally leads to a large coarse system to solve. Therefore, the dimension of the Trefftz space is generally larger than Nicolaides, but we note that the contrast between the dimensions of two spaces reduces as $N$ grows.  In general,  the dimension of the Trefftz coarse space appears reasonable given the geometrical complexity of the computational domain. We also note that the Trefftz space outperforms the Nicolaides space significantly for the domains including walls. The Trefftz space is robust with respect to data complexity, performing similarly with and without the addition of walls in the domain.




\begin{table}
\centering
\caption{ GMRES iterations, dimension, and relative dimension for the Trefftz and Nicolaides coarse spaces.  Results are shown for minimal geometric overlap and $\smallop$.
As the dimension of the Nicolaides space will change with respect to the overlap, its dimension is given as the average dimension over the two overlap values.
}
\setlength{\tabcolsep}{3.8pt}
\begin{tabular}	{|c r |cc|c|cc|c|} 
	\hline
	& & \multicolumn{3}{|c}{Nicolaides}  &  \multicolumn{3}{|c|}{Trefftz}
	\\ 
	\hline
	& &	\multicolumn{2}{|c|}{it.}  &  \multicolumn{1}{|c|}{dim. (rel)}	  &	\multicolumn{2}{|c|}{it.}  &  \multicolumn{1}{|c|}{dim. (rel)}	      \\
	N	&&min. & $\smallop$    &&min &  $\smallop$  & \\
\hline
4& no walls&107&40&18 (1.1)&31&16&144 (5.8)\\
& walls &194&54&39 (2.4)&35&18&283 (11.3)\\
\hline
8& no walls&110&65&74 (1.2)&35&18&328 (4.0)\\
& walls &213&96&150 (2.3)&40&21&658 (8.1)\\
\hline
16& no walls&103&62&300 (1.2)&37&18&797 (2.8)\\
& walls &156&100&504 (2.0)&43&22&1499 (5.2)\\
\hline
32& no walls&94&53&1157 (1.1)&38&18&1997 (1.8)\\
& walls &142&92&1614 (1.6)&41&21&3319 (3.0)\\
\hline
\end{tabular}
\label{tab:table}
\end{table}

\begin{figure}
\centering
\begin{subfigure}{0.495\textwidth}
\centering
\includegraphics[width=\linewidth]{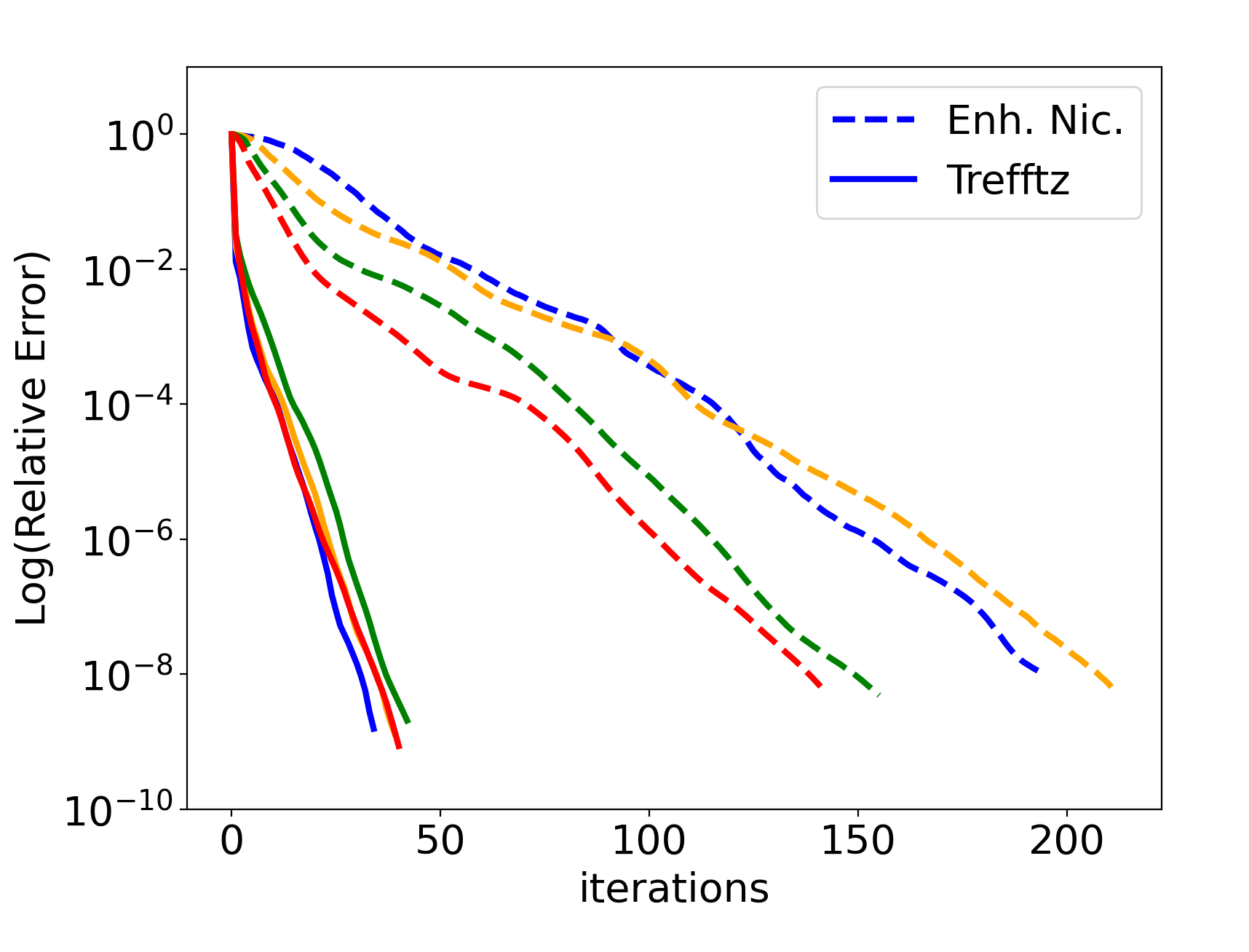}
\caption{Minimal geometric overlap}
\end{subfigure}
\begin{subfigure}{0.495\textwidth}
\centering
\includegraphics[width=\linewidth]{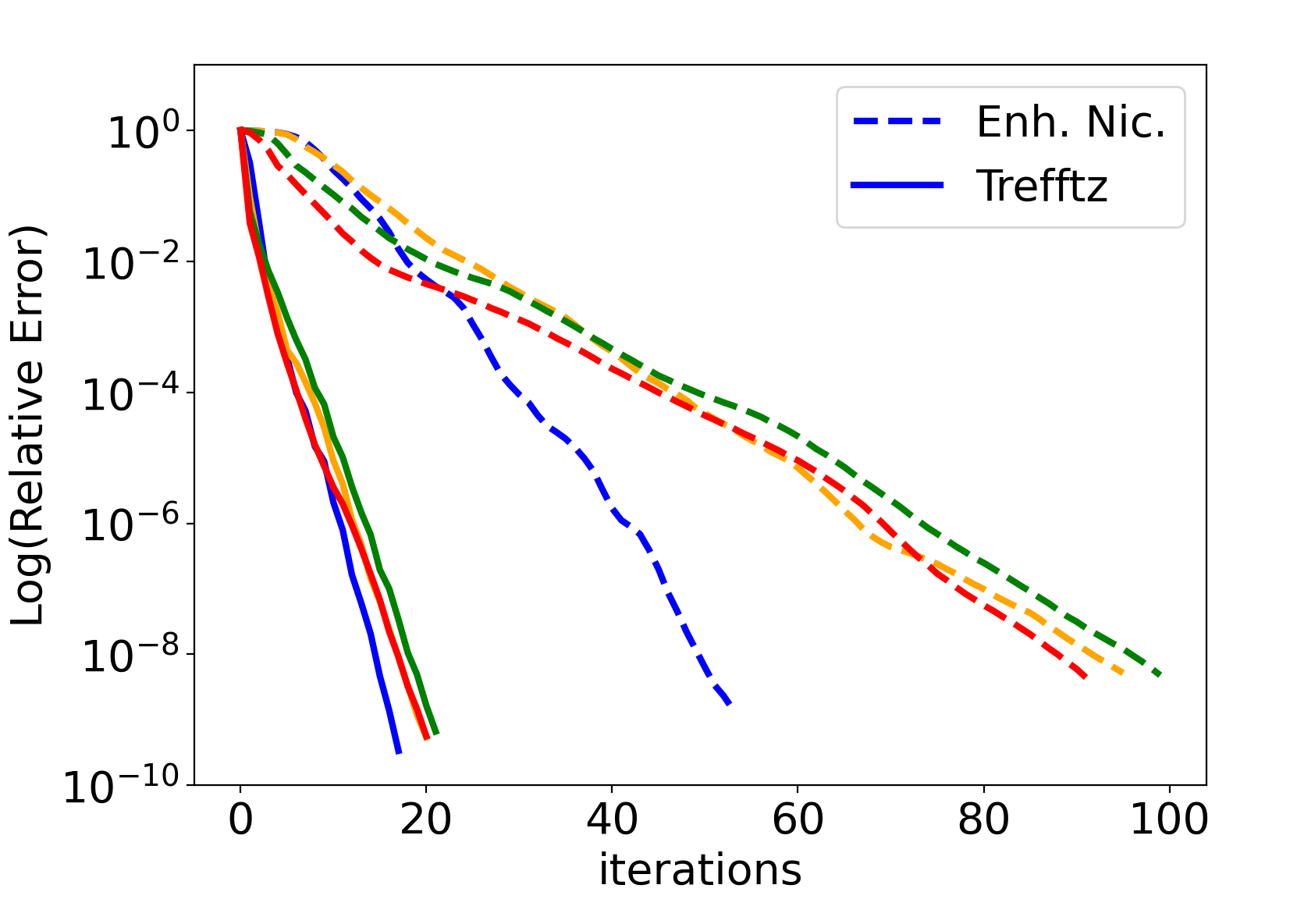}
\caption{Overlap $\smallop$}
\end{subfigure}
\caption{Convergence curves for the Trefftz (solid lines) and Nicolaides (dashed lines) coarse spaces for the larger data set involving both buildings and walls and two overlap sizes. Colors correspond to the number of subdomains as follows: $N=16$ (blue), $N=64$ (orange), $N=256$ (green), $N=1024$ (red).}
\label{fig:convcurves}
\end{figure}

\section{Conclusions}\label{sec:conc}
We presented a theoretical and numerical study of 
the Trefftz coarse space for the Poisson problem posed in domains that include a large number perforations, such as those encountered in the field of urban hydraulics. The main theoretical contribution concerns the error estimate regarding the $H^1$-projection over the coarse space. The error analysis does not rely on global regularity of the solution and is performed under some very minimal assumptions regarding the geometry of the domain. In accordance to the presented error analysis, for a specific edge refinement procedure, the coarse approximation achieves superconvergence. This superconvergence is observed in the numerical experiment involving the L-shaped domain. Combined with RAS, the Trefftz space leads to an efficient and robust iterative solver or preconditioner for linear systems resulting from fine-scale finite element discretizations. The performance of the two-level RAS method is demonstrated through numerical experiments involving realistic urban geometries. We observe that, for finite element discretizations with moderate accuracy, the two-level RAS method reaches the precision of the fine-scale discretization in a few iterations.
Used in combination with domain decomposition methods as a preconditioner for Krylov methods, the coarse space provides significant acceleration in terms of Krylov iteration counts when compared to a more standard Nicolaides coarse space.  This improvement comes at the price of a somewhat larger coarse problem with a larger dimension.
\red{
Future research directions involve extending the presented two-level preconditioning strategy to nonlinear PDEs that model free-surface flows. In the case of nonlinear problems, the coarse space can be used both as a coarse space for the linearized Newton system or as a component of a two-level nonlinear preconditioning strategy.}


\section*{Acknowledgments}

This work has been supported by ANR Project Top-up (ANR-20-CE46-0005).  
The high-resolution structural data has been provided by Métropole
Nice Côte d'Azur.  We warmly thank Florient Largeron,  chief of MNCA's
SIG 3D project,  for his help in preparation of the data and for the multiple fruitful
discussions.

\bibliography{biblioabbrevFV}
\bibliographystyle{model1-num-names}
\end{document}